\documentclass[12pt]{amsart}

\usepackage[english]{babel}
\usepackage[latin1]{inputenc}
 \usepackage[T1]{fontenc}
 \usepackage{lmodern}
 \usepackage{amssymb}
\usepackage{amsmath}
\usepackage{enumerate}

\newtheorem{theo}{Theorem}[section]
\newtheorem{lemm}[theo]{Lemma}
\newtheorem{prop}[theo]{Proposition}
\newtheorem{coro}[theo]{Corollary}
\newtheorem{defi}[theo]{Definition\rm}
\newtheorem{rema}[theo]{Remark}

\def\og{\leavevmode\raise.3ex\hbox{$\scriptscriptstyle\langle\!\langle$~}}
\def\fg{\leavevmode\raise.3ex\hbox{~$\!\scriptscriptstyle\,\rangle\!\rangle$}}

\newcommand{\N}{\mathbb{N}}
\newcommand{\Z}{\mathbb{Z}}

\newcommand{\C}{\mathbb{C}}

\def\germ #1 {\mathfrak{#1}}
\def\cal #1 {\mathcal{#1}}

\title[Restriction and category $\mathcal{O}$]
{Restriction to Levi subalgebras and generalization of the category $\mathcal{O}$}

\author{Guillaume Tomasini}

\address{Institut for Matematiske Fag\\ 
Aarhus Universitet\\
Ny Munkegade 118\\
Bygning 1530\\
8000 Aarhus C}

\email{tomasini@imf.au.dk}

\begin{document}

\begin{abstract}
 The category of all modules over a reductive complex Lie algebra is wild, and therefore it is useful to study full subcategories. For instance, Bernstein, Gelfand and Gelfand introduced a category of modules which provides a natural setting for highest weight modules. In this paper, we define a family of categories which generalizes the BGG category and we classify the simple modules for a subfamily. As a consequence, we show that some of the obtained categories are semisimple.
\end{abstract}

\maketitle


\section{Introduction}


The problem of understanding the restriction of a module over an algebra or a group to a given subalgebra or subgroup is referred to as a branching problem. Branching rules play an important part in representation theory and in physics (e.g. Clebsch-Gordon coefficients). Of great importance to us is the special case of dual pairs. Here one considers a Lie algebra $\germ g $ and a pair of Lie subalgebras $(\germ a , \germ b )$ which are mutual commutants. Then the natural question is to understand the restriction of a (simple) $\germ g $-module to the Lie algebra $\germ a \oplus \germ b $. Such branching rules where obtained first by R. Howe in the case of the metaplectic group and the so-called Weil representation \cite{Ho89a,Ho89b}. Since then, many mathematicians have continued Howe's study of the Weil representation and extended it to others Lie groups or Lie algebras (e.g. \cite{RS95,Pr96,Li99b,Li00}).

Even though we now know many examples, the problem of finding branching rules is highly non trivial. It is even harder if one wants to study non highest weight modules (Note that the Weil representation is from the infinitesimal point of view a highest weight module). To prove some branching rules for a general (simple) weight module, we need to impose some extra conditions. In \cite{Li99b}, Jian-Shu Li studies the restriction of the minimal representation of $E_{7}$ to the dual pair $(A_{1}, F_{4})$. To give a formula in this case, his first step was to use a branching rule from $E_{7}$ to $E_{6}$ which is a Levi subalgebra of $E_{7}$ containing the subalgebra $F_{4}$ appearing in the dual pair.

Motivated by this example, we would like to study the following problem. Let $\germ g $ be a reductive finite dimensional Lie algebra over $\C$. Let $\germ l $ be a Levi subalgebra of $\germ g $.

{\bf Problem $\cal P (\germ l )$:} Can we find all the simple weight $\germ g $-modules $M$ such that the restriction of $M$ to $\germ l $ splits into a direct sum of simple highest weight $\germ l $-modules?

In this article we give a partial answer to this question (and explain why it is only a partial answer).  Our strategy is to add more conditions on the modules $M$ having the above property. More specifically we introduce a family of categories taking into account the above property and some cuspidality condition (see the definition \ref{defiGO}). We then study these categories and their simple modules. In some cases we give a complete description of the simple modules and show that the category is semisimple.

The article is organized as follows. In section 2, we recall some facts about weight modules and their classification. In the third section, we give the definition of our family of categories and give some non trivial examples. In the fourth section, we state and prove a classification result (Theorem \ref{thmpcpl}). Finally in the last section we prove that (some of) our categories are semisimple (Corollaries \ref{corextAf}, \ref{corextC} and \ref{corextCf}).

{\bf Conventions:} All the Lie algebras considered in this paper are finite dimensional and defined over $\C$. We shall denote by $\N=\{0,1,2,\ldots\}$ the set of non negative integers and by $\Z$ the set of all integers. We denote by $\delta_{i,j}$ the Kronecker $\delta$-symbol.

This article is a part of the author's thesis \cite{To}. The main results were announced in \cite{To10}



\section{Weight modules}\label{sec:WM}


The study of the modules over a given Lie algebra lead the mathematicians to explore various categories. The category of all finite dimensional modules was studied first and then it was enlarged to obtain the so-called BGG category $\cal O $ which gave rise later on to the notion of weight modules. Before we state some definitions and review the main results about these modules, some notations are introduced.

Let $\germ g $ denote a reductive Lie algebra and $\cal U (\germ g )$ denote its universal enveloping algebra. Let $\germ h $ be a fixed Cartan subalgebra and denote by $\cal R $ the corresponding set of roots. For $\alpha \in \cal R $, we denote by $\germ g _{\alpha}$ the root space for the root $\alpha$. More generally for $S \subset \cal R $ we denote by $\germ g _{S}$ the direct sum of the root spaces for the various $\alpha \in S$. For $S\subset \cal R $ we denote by $\langle S \rangle$ the set of all roots which are linear combination of elements of $S$. As a particular case, given a basis $\Phi$ of $\cal R $, we consider $\theta\subset \Phi$ and the set of roots $\langle \theta\rangle$. We then consider $\cal R ^{\pm}$ the set of positive (resp. negative) roots with respect to $\Phi$ and $\langle \theta \rangle^{\pm}=\langle \theta \rangle\cap \cal R ^{\pm}$. We define the following subalgebras of $\germ g $ :
$$\germ l _{\theta}=\germ h \oplus \germ g _{\langle \theta\rangle}, \quad \germ n ^{\pm}_{\theta}=\germ g _{\cal R ^{\pm}-\langle \theta\rangle ^{\pm}}.$$

The subalgebra $\germ p _{\theta}=\germ l _{\theta}\oplus \germ n ^+_{\theta}$ is called the standard parabolic subalgebra associated to $\theta$ and $\germ l _{\theta}$ is the standard Levi subalgebra associated to $\theta$. The latter is a reductive algebra. Its semisimple part is denoted by $\germ l '_{\theta}$. If $\theta=\emptyset$, then $\germ l _ {\emptyset}=\germ h $ and we simply write $\germ n ^+ $ instead of $\germ n ^+_{\emptyset}$.


\subsection{The category of weight modules}\label{ssec:WM}


We denote by $Mod(\germ g )$ the category of all $\germ g $--modules. This category is very wild and therefore we will investigate some full subcategories of $Mod(\germ g )$ for which we can describe the simple modules. The first well known subcategory of $Mod(\germ g )$ is the full subcategory of finite dimensional modules $Fin(\germ g )$. It was studied by \'E. Cartan, H. Weyl, W. Killing and many others. This category turns out to be semisimple: every finite dimensional module splits into a direct sum of simple modules. Moreover, we can completely describe the simple objects of this category: they are all the simple highest weight modules with integral dominant highest weight.

We now turn to a bigger category. The category of weight modules.
\begin{defi}\label{def:WM}
A module $M$ is a \emph{weight module} if it is finitely generated, and $\germ h $--diagonalizable in the sense that 
$$M=\oplus_{\lambda \in \germ h ^*}\: M_{\lambda}, \quad \mbox{where } M_{\lambda}=\{m\in M \: : \: H\cdot m = \lambda(H)m, \: \forall \: H\in \germ h \},$$ with weight spaces $M_{\lambda}$ of finite dimension. We will denote by $\cal M (\germ g , \germ h )$ the full subcategory of $Mod(\germ g )$ consisting of all weight modules.
\end{defi}
\begin{rema}
Note that we require finite dimensional weight spaces in our definition, which is not always the case in the literature. This category also appears as a particular case of several other categories (e.g. \cite{PS02,PZ04a} or \cite{DFO94,FMO10}).
\end{rema}

This category has several \og good\fg\ properties, for instance we have the following:
\begin{prop}[Fernando {\cite[thm 4.21]{Fe90}}]\label{prop:JH}
The category $\cal M (\germ g ,\germ h )$ is abelian, noetherian and artinian.
\end{prop}

Before trying to get a better understanding of this category, we will need more concepts that we shall review now.
\begin{defi}\label{def:afin}
Let $\germ a $ be any subalgebra of $\germ g $. A module $M$ is \emph{$\germ a $-finite} if 
$$\forall\: m\in M, \: dim(\cal U (\germ a )m)<\infty.$$
\end{defi}

In the $70's$, Bernstein, Gelfand and Gelfand enlarged the category $Fin(\germ g )$ to include all highest weight modules \cite{BGG75}. More precisely, they introduced the following category:
\begin{defi}\label{def:O}
The category $\cal O $ is the full subcategory of $\cal M (\germ g ,\germ h )$ whose objects are $\germ n ^{+}$-finite.
\end{defi}

This category is quite well understood now. The complete list of simple modules in $\cal O $ is known and there are also important results about projective objects, resolutions\ldots We refer the reader to \cite{Hu08} and the references therein.

In the $80's$, several subcategories of category $\cal O $ where defined and investigated by Rocha-Caridi, the so-called parabolic versions of $\cal O $. We recall here a definition.
\begin{defi}\label{def:Opara}
Let $\germ p $ be a parabolic subalgebra of $\germ g $ and write $\germ p =\germ l \oplus \germ n $ for a Levi decomposition of $\germ p $. The category $\cal O ^{\germ p }$ is the full subcategory of $\cal O $ whose objects $M$ satisfy the following condition:

As a $\germ l $-module, $M$ splits into a direct sum of simple finite dimensional modules.
\end{defi}
The classification of simple modules in $\cal O ^{\germ p }$ and results about projective objects were obtained by Rocha-Caridi in \cite{RoC80}. See also \cite{Hu08}. Note that this is a category in which a certain kind of restriction condition is required.

The next step to understand simple weight modules is the notion of generalised Verma modules. We recall here some well known facts about these modules.

Let $\germ p $ be a parabolic subalgebra of $\germ g $. Let $\germ p = \germ l \oplus \germ n $ be a Levi decomposition of $\germ p $. Given a $\germ l $--module $M$, one can construct a $\germ p $--module structure on $M$ by letting $\germ n $ act trivially. For such a module $M$ we define the \emph{generalised Verma module} $V(\germ p ,M):=\cal U (\germ g )\otimes _{\cal U (\germ p )} \: M.$ Conversely, given any $\germ g $--module $V$ we define the $\germ l $--module $V^{\germ n }:=\{v\in V \: | \: \germ n \cdot v=0\}.$ We then have the following:

\begin{prop}\label{GVM}\hfill{ }
\begin{enumerate}
\item Let $M$ be a weight $\germ l $-module. Then $V(\germ p , M)$ is a weight $\germ g $-module
\item If $M$ is a simple $\germ l $--module, the module $V(\germ p ,M)$ is indecomposable and admits a unique maximal submodule $K(\germ p ,M)$ and a unique simple quotient $L(\germ p ,M)$.
\item Assume $M$ is a simple $\germ l $--module. Then the image of $1\otimes M$ in $L(\germ p ,M)$ is isomorphic to $M$ as $\germ l $-modules.
\item If $V$ is a simple $\germ g $--module such that $V^{\germ n } \not=\{0\}$, then $V\cong L(\germ p ,L^{\germ n })$.
\item Assume $M$ is a simple $\germ l $--module. Let $pr$ be the projection from $V(\germ p ,M)$ onto $1 \otimes M$ given by the decomposition of $V(\germ p ,M)$ into weight spaces. Then the module $K(\germ p , M)$ can be characterized as follows:
$$K(\germ p ,M)=\{v \in V\: :\: \forall \: u \in \cal U (\germ n ), \: pr(u\cdot v)=0\}.$$
\end{enumerate}
\end{prop}
\begin{proof} We refer to \cite[proposition $3.8$]{Fe90} for a proof.

\end{proof}

The module $L(\germ p ,M)$ will be called the simple $\germ g $--module \emph{induced} from $(\germ p ,M)$. We refer to \cite{CF94,Maz00} for a more detailed discussion about generalised Verma modules.

To give the classification of simple weight modules, we need one more ingredient: the so--called \emph{cuspidal} modules.
\begin{defi}\label{def:cusp}
Let $M$ be a weight module. A root $\alpha \in \cal R $ is said to be \emph{locally nilpotent} with respect to $M$ if any non zero $X\in \germ g _{\alpha}$ acts by a locally nilpotent operator on the whole module $M$. It is said to be \emph{cuspidal} if any non zero $X\in \germ g _{\alpha}$ acts injectively on the whole module $M$.
\end{defi}

We denote by $\cal R ^N(M)$ the set of locally nilpotent roots and by $\cal R ^I(M)$ the set of cuspidal roots. We shall simply denote them by $\cal R ^N $ and $\cal R ^I$ when the module $M$ is clear from the context. Then we have the following:
\begin{lemm}[Fernando, {\cite[lemma $2.3$]{Fe90}}]
For a simple weight module $M$, $\cal R = \cal R ^N\sqcup \cal R ^I$.
\end{lemm}

\begin{defi}\label{def:CM}
Let $S \subset \cal R $. A weight module is called \emph{$S$-cuspidal} if $S\subset \cal R ^I$. When $S=\cal R $, it is simply called \emph{cuspidal}.
\end{defi}
\begin{rema}
If $M$ is a simple cuspidal weight module, then all its weight spaces have the same dimension and for all $\alpha \in \cal R $ and any non zero $X\in \germ g _{\alpha}$, $X$ acts bijectively on $M$.
\end{rema}

Set $\cal R ^N_s=\{\alpha \in \cal R ^N \: : \: -\alpha \in \cal R ^N\}, \quad \cal R ^N_a=\cal R ^N\setminus\cal R ^N_s$. We define $\cal R ^I_s$ and $\cal R ^I_a$ the same way. Recall the following theorem:

\begin{theo}[Fernando {\cite[theorem $4.18$]{Fe90}}, Futorny \cite{Fu}]\label{thmFe}%
Let $M$ be a simple weight module. Then there are a parabolic subalgebra $\germ p $ of $\germ g $ with a Levi decomposition $\germ p =\germ l \oplus \germ n $, and a simple cuspidal $\germ l $-module $C$ such that $M\cong L(\germ p ,C)$.
\end{theo}
\begin{rema}
More precisely there are a basis $\Phi$ of $\cal R $ and a subset $\theta\subset \Phi $ such that $\langle \theta \rangle =\cal R ^I_s$ and $\cal R ^+\setminus\langle \theta \rangle ^+\subset \cal R ^N$. Then $M^{\germ n ^+_{\theta}}$  is a simple cuspidal $\germ l _{\theta}$-module and $M\cong L(\germ p _{\theta},M^{\germ n ^+_{\theta}})$.
\end{rema}

The theorem of Fernando reduces the classification of simple weight $\germ g $--modules to the classification of simple cuspidal weight modules for reductive Lie algebras. By standard arguments this reduces to the classification of simple cuspidal modules for simple Lie algebras. A first step towards this classification is given by the following theorem:

\begin{theo}[Fernando {\cite[theorem $5.2$]{Fe90}}]\label{thmFeC}
Let $\germ g $ be a simple Lie algebra. If $M$ is a simple cuspidal $\germ g $--module, then $\germ g $ is of type $A$ or $C$
\end{theo}

The classification of simple cuspidal modules was then completed in two steps. In the first step Britten and Lemire classified all simple cuspidal modules of degree $1$ (see \cite{BL87}), where $deg(M)=\sup_{\lambda \in \germ h ^*} \: \{dim(M_{\lambda})\}.$ We will come back to these modules later on as they will play an important part in our study. Then  Mathieu gave the full classification of simple cuspidal modules of finite degree by introducing the notion of a coherent family (see \cite{Ma00}).


\subsection{The case of $\mathfrak{sl}_{2}$}\label{ssec:sl2}


In this section, we review the classification of weight modules for $\germ g =\germ sl _{2}$. We shall fix an $\germ sl _{2}$-triple $(X^-,H,X^+)$. We therefore have the following relations:
$$[H,X^{\pm}]=\pm 2X^{\pm}, \quad [X^+,X^-]=H.$$

Recall that the degree of a weight module $M$ is 
$$deg(M)=\sup\{dim(M_{\lambda}),\: \lambda \in \germ h ^*\}.$$

\begin{prop}\label{prop:WMsl2}
Let $M$ be a simple weight $\germ sl _{2}$-module. Then $deg(M)=1$.
\end{prop}
\begin{proof}
Recall that $\Omega = \frac{1}{2}H^2+H+2X^-X^+$ is in the center of the universal enveloping algebra of $\germ sl _{2}$. Therefore, $M$ being simple, $\Omega$ acts as a scalar operator. On the other hand, as $M$ is a weight module, $H$ acts on each weight space by some constant (the weight). Therefore, on each weight space, $X^-X^+$ acts by some constant. From this, we conclude that $\cal U (\germ g )_{0}$, the commutant of $\C H$, acts by some constant on each weight space. But, since $M$ is simple, given two non zero vectors $v$ and $w$ in the same weight space, there should exist some element $u\in \cal U (\germ g )$ sending $v$ to $w$. The fact that $v$ and $w$ have the same weight forces $u$ to be in the commutant of $\C H$. From the above we know that $u$ acts by some constant. This forces $v$ and $w$ to be proportional and therefore the corresponding weight space is $1$-dimensional. This completes the proof.

\end{proof}

Now we recall the construction of simple cuspidal $\germ g $-modules. Let $a=(a_{1},a_{2})\in \C^2$. Assume that $a_{1}$ and $a_{2}$ are not integers. We construct then a vector space as follows. For each $k\in \Z$, we define a vector $x(k)$. The vector space generated by these (formal) elements is denoted by $N(a)$. We put the following action of $\germ g $ on $N(a)$:
$$\left\{\begin{array}{ccc}
H\cdot x(k) & = & (a_{1}-a_{2}+2k)x(k),\\
X^+\cdot x(k) & = & (a_{2}-k)x(k+1),\\
X^-\cdot x(k) & = & (a_{1}+k)x(k-1).
\end{array}\right.$$
It is now easy to see that $N(a)$ is a simple cuspidal $\germ g $-module. It turns out that any simple cuspidal $\germ sl _{2}$-module is of this form (see \cite{BL87}).


\subsection{The category of cuspidal modules}\label{ssec:CM}


The category of all cuspidal modules has been intensively studied since Mathieu's classification result. We will not try to recall all the known results here. We refer the reader to \cite{GS06,GS07,MS10,MS10b} for details.

In the last part of this article we will be interested in extension between modules. We review now some facts about that. Given two weight $\germ g $-modules $M$ and $N$, one wants to find all the weight modules $V$ such that the sequence $0\rightarrow N \rightarrow V \rightarrow M\rightarrow 0$ is exact. This problem can be solved by using cocycles. We recall its definition:
\begin{defi}\label{def:cocy}
Let $M$ and $N$ be two weight modules. A \emph{cocycle} from $M$ to $N$ is a linear map $c:\germ g \rightarrow Hom_{\C}(M,N)$ such that: For every $m\in M$, we have $$c([X,Y])(m)=[c(X),Y](m)+[X,c(Y)](m),$$ where the bracket in the right hand side is the commutator in $Hom_{\C}(M,N)$, for instance $[c(X),Y](m)=c(X)(Y\cdot m)-Y\cdot (c(X)(m))$.
\end{defi}
Given a cocyle $c$ from $M$ to $N$, one can construct a $\germ g $-module structure on $V:=N\oplus M$ as follows:
For any $X\in \germ g $, define $X\cdot (n,m):=(X\cdot n+c(X)(m),X\cdot m)$. Then it is easy to see that endowed with this structure $V$ fits into an exact sequence $0\rightarrow N \rightarrow V \rightarrow M\rightarrow 0$. Moreover, for every exact sequence $0\rightarrow N \rightarrow V \rightarrow M\rightarrow 0$ there is a cocycle $c$ from $M$ to $N$ such that $V$ is isomorphic (as a vector space) to $N\oplus M$ with the $\germ g $-module structure given as above. Besides, such an exact sequence splits exactly when there is a linear map $\phi : M\rightarrow N$ such that $c(X)=[X,\phi]$ for all $X \in \germ g $. Such a cocycle is called a \emph{coboundary}.

The case when $\germ g =\germ sl _{2}$ is particularly simple. It is given in the following:

\begin{prop}[Grantcharov, Serganova {\cite[example 3.3]{GS07}}]\label{prop:GS}
Let $M$ and $N$ be simple cuspidal $\germ sl _{2}$--modules. If $M\not\cong N$, then every cocycle from $M$ to $N$ is a coboundary. If $M=N$, then up to a coboundary, every cocycle $c$ from $M$ to $N$ has the following form: $c(H)=0=c(X^-)$ and $c(X^+)=b\times {\left(X^-\right)}^{-1}$ where $b\in \C$.
\end{prop}
\begin{rema}
Note that as $M$ is a simple cuspidal module, $X^-$ (and $X^+$) acts bijectively on $M$. Thus the operator $X^-$ has an inverse which we denote by $(X^-)^{-1}$.
\end{rema}

The general case for $\germ sl _{n}$ is more complicated. As we will not need it, we do not mention it here (see \cite{GS07,MS10b}). On the other hand we will need the case of $\germ sp _{2n}$. We recall the following:

\begin{theo}[Britten, Khomenko,Lemire, Mazorchuk {\cite[thm 1]{BKLM04}}]\label{thmBKLM}
The category of cuspidal $\germ sp _{2n}$-module is semisimple. 
\end{theo}
This theorem means that every cocycle between two simple cuspidal $\germ sp _{2n}$-modules is a coboundary.


\section{The category $\mathcal O _{S,\theta}$}\label{sec:GO}


In all this section, $\germ g $ denote a reductive Lie algebra and $\germ h $ a fixed Cartan subalgebra. We also denote by $\cal R $ the set of roots of $(\germ g ,\germ h )$.


\subsection{General definition}\label{ssec:GGO}


The general definition requires some subsets of $\cal R $. We first recall some basic definitions. Given a subset $S$ of $\cal R $ we denote $S_{s}$ its symmetric part: $S_{s}=S\cap -S$ and $S_{a}$ its antisymmetric part: $S_{a}=S\setminus S_{s}$.

\begin{defi}[see \cite{Bo81}]\label{def:root}
A subset $S$ of $\cal R $ is \emph{symmetric} if $S=-S$ ; it is \emph{closed} if the conditions $\alpha\in S, \: \beta \in S, \: \alpha+\beta \in \cal R $ imply $\alpha+\beta\in S$. A \emph{parabolic} subset of $\cal R $ is a closed subset $P$ such that $P\cup -P=\cal R $. A \emph{Levi} subset of $\cal R $ is a closed and symmetric subset of $\cal R $.
\end{defi}
\begin{rema}
Note that if $P$ is a parabolic subset, $P_{s}$ is a Levi subset. In this case we call $P_{s}$ the \emph{Levi part} of $P$. The antisymmetric part $P_{a}$ of $P$ should also be referred to as the \emph{unipotent part} of $P$.
\end{rema}
Given a Levi subset $S$, we denote $\germ l _{S}:=\germ h \oplus \germ g _{S}$. Given a parabolic subset $P$, we denote $\germ n ^+_{P}:=\germ g _{P_{a}}$ and $\germ p _{P}:= \germ l _{P_{s}}\oplus \germ n ^+_{P}$. For any subset $S$ of $\cal R $, we denote $Q_{S}$ the lattice generated by $S$.

\begin{defi}\label{defiGO}
Let $S$ and $T$ be two Levi subsets of $\cal R $ such that $Q_{S}\cap Q_{T}=0$. Let $P$ be a parabolic subset containing $S\cup T$ and let $B$ be a basis of $T$. We denote by $\cal O _{P,S,T,B}$ the full subcategory of the category of weight $\germ g $-modules $M$ such that
\begin{enumerate}
\item The module $M$ is $S$-cuspidal,
\item As a $\germ l  _{T}$-module, $M$ splits into a direct sum of simple $B$-highest weight modules,
\item The module $M$ is $P_{a}$-finite.
\end{enumerate}
\end{defi}

\begin{rema}
The condition $Q_{S}\cap Q_{T}=0$ may seem strange. In fact, if this condition is not fulfilled then the corresponding category (defined the same way) would only consist of the zero module.
\end{rema}

Let $T$ be some Levi subset of $\cal R $. Obviously, any simple weight $\germ g $-module having the restriction property $\cal P (\germ l _{T})$ would be in such a category. It suffices to choose $S=\cal R ^{I}_{s}(M)$, $B$ to be the basis corresponding to the restriction property of $M$ and $P=\cal R $. Of course, it could happen that some smaller choice for $P$ is also possible. But this situation is in some sense the basic one, as asserted in the following

\begin{prop}\label{prop:GGO}
Let $(P,S,T,B)$ be as in definition \ref{defiGO}. Assume $P_{a}\not= \emptyset$. Then every simple $\germ g $-module $M$ in $\cal O _{P,S,T,B}$ is of the form $L(\germ p ,N)$ where $\germ p =\germ l _{P_{s}}\oplus \germ g _{P_{a}}$ and $N$ is a simple module in $\cal O _{P_{s},S,T,B}(\germ l _{P_{s}})$.
\end{prop}
\begin{proof}
Apply proposition \ref{GVM} to the module $M$ and to the parabolic algebra $\germ p =\germ l _{P_{s}}\oplus \germ g _{P_{a}}$. Here the third condition in the definition \ref{defiGO} ensures that $M^{\germ g _{P_{a}}}\not= \{0\}$.

\end{proof}


\subsection{Particular case}\label{ssec:GO}


Unfortunately, the general case of category $\cal O _{P,S,T,B}$ does not seem easy to study. Thus, in the sequel we would be interested in a particular case of this general definition. For sake of clarity, we write down explicitly its definition. Let us fix a basis $\Phi$ of $\cal R $.

\begin{defi}\label{def:GO}
Let $\theta\subset S\subset \Phi$.  We denote by $\cal O _{S,\theta}(\germ g )$ or simply by $\cal O _{S,\theta}$ the full subcategory of the category of weight $\germ g $-modules $M$ such that
\begin{enumerate}
\item The module $M$ is $\langle S\setminus \theta\rangle$-cuspidal,
\item As a $\germ l _{\theta}$-module, $M$ splits into a direct sum of simple highest weight modules,
\item The module $M$ is $\germ n ^{+}_{S}$-finite.
\end{enumerate}
\end{defi}

In other terms, we have $\cal O _{S,\theta}=\cal O _{\langle S\rangle \cup \cal R ^+, \langle S\setminus \theta\rangle, \langle \theta\rangle,  \theta}$. Of course, not every category $\cal O _{P,S,T,B}$ is of this latter form. In what follows we shall refer to the first property as the {\bf cuspidality condition} and to the second one as the {\bf restriction condition}.

Note that if $S=\theta=\emptyset$, then we recover the usual category $\cal O $ of Bernstein-Gelfand-Gelfand. More generally, when $S=\theta$ we get a generalisation of the category $\cal O ^{\germ p _{_{S}}}$ of Rocha-Caridi. Indeed, remember that category $\cal O ^{\germ p _{_{S}}}$ also requires that the $\germ l _{S}$-highest weight modules have finite dimension. We shall see later on that we could not impose such a strong condition on our category (see proposition \ref{prop:GO}). Finally if $\theta=\emptyset$ and $S=\Phi$ then we recover the category of all cuspidal modules.


\subsection{First properties}\label{ssec:GOprop}


Let us now examine the first easy properties that carry the categories $\cal O _{S,\theta}$. We have the following:

\begin{prop}\label{prop:GO}
Let $\theta \subset S\subset \Phi$. Then:
\begin{enumerate}
\item The category $\cal O _{S,\theta}$ is abelian, noetherian and artinian.
\item The multiplicities of the $\germ l _{\theta}$--modules appearing in the decomposition of a module $M\in \cal O _{S,\theta}$ are finite.
\item Assume there exists $\alpha \in S\setminus\theta$ and $\beta \in \theta$ such that $\alpha+\beta\in \cal R $. Then the simple $\germ l _{\theta}$--modules in the decomposition of any simple module in $\cal O _{S,\theta}$ are of infinite dimension.
\item The simple modules in $\cal O _{S,\theta}(\germ g )$ are of the form $L(\germ p _{S},N)$ where $N$ is a simple module in $\cal O _{S,\theta}(\germ l _S)$.
\end{enumerate}
\end{prop}
\begin{proof}
\begin{enumerate}
\item Thanks to proposition \ref{prop:JH}, we only need to check that the category $\cal O _{S,\theta}$ is stable by finite direct sums, taking submodules and quotients. Everything here is obvious except the cuspidality condition for a quotient. Therefore, let $M$ be in $\cal O _{S,\theta}$ and $N$ be a proper submodule of $M$. We prove that $M/N$ satisfies the cuspidality condition. First note that $\oplus_{\lambda}\: (M_{\lambda}+N)/N$ is a weight space decomposition of $M/N$. Note also that $(M_{\lambda}+N)/N\oplus N_{\lambda}=M_{\lambda}$. Let $\alpha \in \langle S\setminus\theta\rangle$ and let $X$ be a non zero vector in $\germ g _{\alpha}$. By the cuspidality condition for $M$ and $N$, we have $dim(N_{\lambda})=dim(X\cdot N_{\lambda})$ and $dim(M_{\lambda})=dim(X\cdot M_{\lambda})$. Therefore, we have $dim((M_{\lambda}+N)/N)=dim(X\cdot ((M_{\lambda}+N)/N)$. For this to hold for any $\lambda$ it is necessary that $X$ acts injectively on $(M_{\lambda}+N)/N$.
\item This is an immediate consequence of the fact that the weight spaces are finite dimensional.
\item By the cuspidality condition we have $\alpha \in \cal R ^{I}_{s}$. If there were in the decomposition of $M$ a finite dimensional $\germ l _{\theta}$-module then any root in $\langle \theta\rangle$ should be locally nilpotent. Thus we would have $\beta \in \cal R ^N_{s}$. This together with the hypothesis $\alpha+\beta\in \cal R $ contradicts \cite[lem 4.7]{BBL97}.
\item This is a particular case of proposition \ref{prop:GGO}.
\end{enumerate}

\end{proof}
\begin{rema}
From $(4)$ in the above proposition, we see that one should first study the category $\cal O _{\Phi,\theta}$. However note that $(4)$ does not imply that any simple module induced from a simple module in some $\cal O _{S,\theta}(\germ l _{S})$ is in the category $\cal O _{S,\theta}(\germ g )$.
\end{rema}


\subsection{The modules of degree $1$}\label{ssec:deg1}


So far, we have not shown that at least some new category $\cal O _{S,\theta}$ is non trivial. We will do this now by exhibiting very special modules. These are the infinite dimensional modules of degree $1$. They were introduced and classified by Benkart, Britten and Lemire in \cite{BBL97}. In particular such modules only exist for Lie algebras of type $A$ or $C$. Let us review their construction.

\subsubsection{Modules over the Weyl algebra}\label{sssec:Weyl}

Let $N$ be a positive integer. Recall that the Weyl algebra $W_{N}$ is the associative algebra generated by the $2N$ generators $\{q_{i}, \: p_{i}, \: 1\leq i\leq N\}$ submitted to the following relations:
$$[q_{i},q_{j}]=0=[p_{i},p_{j}],\quad [p_{i},q_{j}]=\delta_{i,j}\cdot 1,$$ where the bracket is the usual commutator for associative algebras.

Define a vector space as follows. Fix some $a\in \C^N$. Let $$\cal K =\{k\in \Z^N \: :\: \mbox{if } a_{i}\in \Z, \mbox{ then } a_{i}+k_{i}<0 \iff a_{i}<0\}.$$ Now our vector space $W(a)$ is the $\C$-vector space whose basis is indexed by $\cal K $. For each $k\in \cal K $, we fix a vector basis $x(k)$. Define an action of $W_{N}$ by the following recipe:
$$\begin{array}{ccc}
q_{i}\cdot x(k) & = & \left\{ \begin{array}{cc} (a_{i}+k_{i}+1)x(k+\epsilon_{i}) & \mbox{if } a_{i}\in \Z_{<0}\\ x(k+\epsilon_{i}) & \mbox{otherwise} \end{array}\right. ,\\
p_{i}\cdot x(k) & = & \left\{ \begin{array}{cc} x(k-\epsilon_{i}) & \mbox{if } a_{i}\in \Z_{<0}\\ (a_{i}+k_{i})x(k-\epsilon_{i}) & \mbox{otherwise} \end{array}\right.
\end{array}$$

Then we have:
\begin{theo}[Benkart,Britten,Lemire {\cite[thm 2.9]{BBL97}}]\label{thmBBLWeyl}
Let $a\in \C^N$. Then $W(a)$ is a simple $W_{N}$-module.
\end{theo}

\subsubsection{Type A case}\label{sssec:deg1A}

In this section only, $\germ g $ denotes a simple Lie algebra of type $A$. We shall construct weight $\germ g $-modules of degree $1$ by using the previous construction. We realize the Lie algebra $\germ g $ inside some $W_{N}$. Let $N-1$ be the rank of $\germ g $. Then, we can embed $\germ g $ into $W_{N}$ as follows: to an elementary matrix $E_{i,j}$ we associate the element $q_{i}p_{j}$ of $W_{N}$. This is easily seen to define an embedding of $\germ g $ into $W_{N}$.  Let $\cal K _{0}=\{k\in \cal K \: :\: \sum_{i=1}^N\: k_{i}=0\}$. Let $N(a)$ be the subspace of $W(a)$ whose basis is indexed by $\cal K _{0}$. Then we have the following:

\begin{theo}[Benkart,Britten,Lemire {\cite[thm 5.8]{BBL97}}]\label{thmBBLA}\hfill{ }
\begin{enumerate}
\item The vector subspace $N(a)$ of $W(a)$ is a simple weight $\germ g $--module of degree $1$. Moreover, $N(a)$ is cuspidal if and only if $a_{i}\not\in\Z$ for all $i\in\{1,\ldots ,N\}$.
\item Conversely if $M$ is an infinite dimensional simple weight $\germ g $--module of degree $1$, then there exist $a=(a_1,\ldots, a_N)\in \C^N$, and two integers $j$ and $l$ such that 
\begin{itemize}
\item[$\bullet$] $a_i=-1 \mbox{ for } i=1,\ldots ,j-1,$
\item[$\bullet$] $a_i\in \C\setminus \Z \mbox{ for } i=j,\ldots ,l,$
\item[$\bullet$] $a_i=0 \mbox{ for } i=l+1,\ldots ,N,$
\item[$\bullet$] and the module $M$ is isomorphic to $N(a)$.
\end{itemize}
\end{enumerate}
\end{theo}

Denote be $\Phi$ the standard basis for the root system of $\germ g $ with respect to the standard Cartan subalgebra $\germ h $. Let $a\in \C^{N}$ such that $a=(\underbrace{-1,\ldots ,-1}_{j},a_{j+1},\ldots, a_m,\underbrace{0,\ldots ,0}_{l})$ where $l+m=N$ and $m>j+1$. Let $\theta_{a}\subset \Phi$ be given by the non-circled simple roots of the following Dynkin diagram:
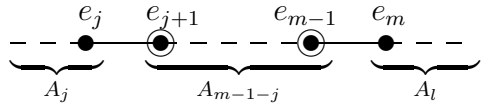
\setlength{\unitlength}{1cm}
\begin{center}
\begin{picture}(8,1)
\put(1,0.5){\line(1,0){.2}}
\put(1.4,0.5){\line(1,0){.2}}
\put(1.8,0.5){\line(1,0){.2}}
\put(2,0.5){\circle*{.2}}
\put(2,0.5){\line(1,0){1}}
\put(3,0.5){\circle*{.2}}
\put(3,0.5){\line(1,0){.2}}
\put(3.4,0.5){\line(1,0){.2}}
\put(3.8,0.5){\line(1,0){.2}}
\put(4.2,0.5){\line(1,0){.2}}
\put(4.6,0.5){\line(1,0){.2}}
\put(5,0.5){\circle*{.2}}
\put(5,0.5){\line(1,0){1}}
\put(6,0.5){\circle*{.2}}
\put(6,0.5){\line(1,0){.2}}
\put(6.4,0.5){\line(1,0){.2}}
\put(6.8,0.5){\line(1,0){.2}}
\put(3,0.5){\circle{.35}}
\put(5,0.5){\circle{.35}}
\put(1.9,0.8){$e_j$}
\put(2.8,0.8){$e_{j+1}$}
\put(4.5,0.8){$e_{m-1}$}
\put(5.8,0.8){$e_{m}$}
\put(1,.3){$\underbrace{\hphantom{111111}}_{A_j}$}
\put(2.8,.3){$\underbrace{\hphantom{111111111111}}_{A_{m-1-j}}$}
\put(5.8,.3){$\underbrace{\hphantom{1111111}}_{A_l}$}
\end{picture}
\end{center}

For the commodity of the reader, we explicit the action of $\germ h $ and of $X_{\pm e_{i}}$ on $N(a)$:
$$\left\{\begin{array}{lcl}
H_{e_{i}}\cdot x(k) & = & (k_{i}-k_{i+1})x(k), \: i=1,\ldots ,j-1\\
H_{e_{j}}\cdot x(k) & = & (-1-a_{j+1}+k_{j}-k_{j+1})x(k)\\
H_{e_{i}}\cdot x(k) & = & (a_{i}-a_{i+1}+k_{i}-k_{i+1})x(k), \: i=j+1,\ldots,m-1\\
H_{e_{m}}\cdot x(k) & = & (a_{m}+k_{m}-k_{m+1})x(k)\\
H_{e_{i}}\cdot x(k) & = & (k_{i}-k_{i+1})x(k), \: i\geq m+1\\
X_{e_{i}}\cdot x(k) & = & k_{i}x(k-\epsilon_{i+1}+\epsilon_{i}), \: i=1,\ldots ,j-1\\
X_{e_{j}}\cdot x(k) & = & k_{j}(a_{j+1}+k_{j+1})x(k-\epsilon_{j+1}+\epsilon_{j})\\
X_{e_{i}}\cdot x(k) & = & (a_{i+1}+k_{i+1})x(k-\epsilon_{i+1}+\epsilon_{i}), \: i=j+1,\ldots,m-1\\
X_{e_{i}}\cdot x(k) & = & k_{i+1}x(k-\epsilon_{i+1}+\epsilon_{i}), \: i\geq m\\
X_{-e_{i}}\cdot x(k) & = & k_{i+1}x(k-\epsilon_{i}+\epsilon_{i+1}), \: i=1,\ldots j-1\\
X_{-e_{j}}\cdot x(k) & = & x(k-\epsilon_{j}+\epsilon_{j+1})\\
X_{-e_{i}}\cdot x(k) & = & (a_{i}+k_{i})x(k-\epsilon_{i}+\epsilon_{i+1}), \: i=j+1,\ldots,m\\
X_{-e_{i}}\cdot x(k) & = & k_{i}x(k-\epsilon_{i}+\epsilon_{i+1}), \: i\geq m+1
\end{array}\right.$$

Now we will prove the following:
\begin{theo}\label{thmNO}
The module $N(a)$ is a simple object of the category $\cal O _{\Phi,\theta_{a}}$. Moreover, the highest weight vectors for the action of $\germ l _{\theta_{a}}$ are the linear combination of the $x(0,\ldots ,0,k_1,\ldots ,k_m,0,\ldots ,0)$ where $k_i \in \Z$ are such that $\sum_i  \: k_i=0$.
\end{theo}
\begin{proof}
From the explicit action of $\germ g $, we easily derive that $N(a)$ is a weight $\germ g $-module which is $\langle\Phi\setminus \theta_{a}\rangle$-cuspidal. Using once again the action of $\germ g $, one checks that the vectors $x(0,\ldots ,0,k_1,\ldots ,k_m,0,\ldots ,0)$ are $\germ l _{\theta_{a}}$-highest weight vectors and that every $\germ l _{\theta_{a}}$-highest weight vector is a linear combination of these vectors.

Then we show that each of these vectors generate a simple $\germ l _{\theta_{a}}$-module. Indeed, we already know that it generates an indecomposable module (since it is a highest weight module). To show it is simple we only have to show that it does not contain any other highest weight vector (since a submodule of a highest weight module is again a highest weight module). As we already know the complete list of highest weight vectors in $N(a)$ we just have to check that the $\germ l _{\theta_{a}}$-module generated by $x(0,\ldots ,0,k_1,\ldots ,k_m,0,\ldots ,0)$ cannot contain the  $\germ l _{\theta_{a}}$-module generated by $x(0,\ldots ,0,k'_1,\ldots ,k'_m,0,\ldots ,0)$ for $(k_{1},\ldots,k_{m})\not=(k'_{1},\ldots,k'_{m})$. Assume it is not the case. Then the action of the center of $\germ l _{\theta_{a}}$ should be the same on these two vectors and the $\germ l '_{\theta_{a}}$-weight of $x(0,\ldots ,0,k'_1,\ldots ,k'_m,0,\ldots ,0)$ should be smaller than that of $x(0,\ldots ,0,k_1,\ldots ,k_m,0,\ldots ,0)$. These two conditions can only be fulfilled in case $(k_{1},\ldots,k_{m})=(k'_{1},\ldots,k'_{m})$. This contradiction completes the proof.

\end{proof}

\begin{rema}
In fact the module $N(a)$ is in general an object of several categories $\cal O _{S,\theta}$.
\end{rema}

\subsubsection{Type C case}\label{sssec:deg1C}

In this section only, $\germ g $ denotes a simple Lie algebra of type $C$. We shall construct weight $\germ g $-modules of degree $1$ in the same way as above. So we need to realize the Lie algebra $\germ g $ inside some $W_{N}$. Let $N$ be the rank of $\germ g $. Then, $span_{\C}\{q_{i}p_{j}, p_{i}p_{j}, q_{i}q_{j}, \: 1\leq i,j\leq N\}$ is a subalgebra of $W_{N}$ isomorphic to $\germ g $. More specifically, the Cartan subalgebra is given by $span\left(\{q_{i}p_{i}-q_{i+1}p_{i+1}, \: i=1,\ldots,n-1\}\cup\{q_{n}p_{n}+\frac{1}{2}\}\right)$, the $n-1$ weight vectors corresponding to the short simple roots are given by $q_{i}p_{i+1}$ with $i=1,\ldots ,n-1$, and the weight vector corresponding to the long simple root is given by $\frac{1}{2}q_{n}^2$. Note that this is not the same kind of embedding as for Lie algebras of type A. 

Let $\cal K _{\bar 0}=\{k\in \cal K \: :\: \sum_{i=1}^N\: k_{i}\in 2\Z\}$. Let $M(a)$ be the subspace of $W(a)$ whose basis is indexed by $\cal K _{\bar 0}$.Then we have the following:

\begin{theo}[Benkart,Britten,Lemire {\cite[thm 5.21]{BBL97}}]\label{thmBBLC}\hfill{ }
\begin{enumerate}
\item The vector subspace $M(a)$ of $W(a)$ is a simple weight $\germ g $--module of degree $1$. Moreover, $M(a)$ is cuspidal if and only if $a_{i}\not\in\Z$ for all $i\in\{1,\ldots ,N\}$.
\item Conversely if $M$ is an infinite dimensional simple weight $\germ g $--module of degree $1$, then there exist two integers $l$ and $m$ with $l+m=N$ and $a=(\underbrace{-1,\ldots ,-1}_{l},a_{1},\ldots , a_{m})$ such that $M\cong M(a)$. Moreover, if $m>1$ then $a_{1},\ldots,a_{m}$ are non integer complex numbers, and if $m=1$ then $a_{1}$ is either a non integer complex numbers or equals to $-1$ or $-2$.
\end{enumerate}
\end{theo}

Denote be $\Phi$ the standard basis for the root system of $\germ g $ with respect to the standard Cartan subalgebra $\germ h $. Let $a\in \C^{N}$ such that $a=(\underbrace{-1,\ldots ,-1}_{l},a_{l+1},\ldots, a_n)$ where $0<l<n$. Let $\theta_{a}\subset \Phi$ be given by the non-circled simple roots of one of the following Dynkin diagram (according to $l=n-1$ or $l<n-1$):
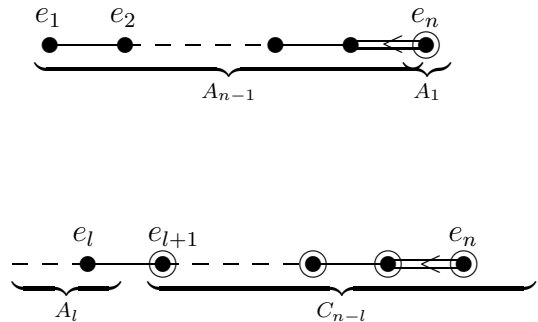
\setlength{\unitlength}{1cm}
\begin{center}
\begin{picture}(7,1)
\put(1,0.5){\circle*{.2}}
\put(1,0.5){\line(1,0){1}}
\put(2,0.5){\circle*{.2}}
\put(2,0.5){\line(1,0){.2}}
\put(2.4,0.5){\line(1,0){.2}}
\put(2.8,0.5){\line(1,0){.2}}
\put(3.2,0.5){\line(1,0){.2}}
\put(3.6,0.5){\line(1,0){.2}}
\put(4,0.5){\circle*{.2}}
\put(4,0.5){\line(1,0){1}}
\put(5,0.5){\circle*{.2}}
\put(5,.45){\line(1,0){1}}
\put(5,.55){\line(1,0){1}}
\put(6,.5){\circle*{.2}}
\put(6,.5){\circle{.35}}
\put(0.8,.8){$e_{1}$}
\put(1.8,.8){$e_2$}
\put(5.8,.8){$e_n$}
\put(5.4,.4){$<$}
\put(0.8,.3){$\underbrace{\hphantom{1111111111111111111111111}}_{A_{n-1}}$}
\put(5.7,.3){$\underbrace{\hphantom{l}}_{A_1}$}
\end{picture}
\end{center}

\setlength{\unitlength}{1cm}
\begin{center}
\begin{picture}(8,1.6)
\put(1,0.5){\line(1,0){.2}}
\put(1.4,0.5){\line(1,0){.2}}
\put(1.8,.5){\line(1,0){.2}}
\put(2,.5){\circle*{.2}}
\put(2,.5){\line(1,0){1}}
\put(3,.5){\circle*{.2}}
\put(3,.5){\line(1,0){.2}}
\put(3.4,.5){\line(1,0){.2}}
\put(3.8,.5){\line(1,0){.2}}
\put(4.2,.5){\line(1,0){.2}}
\put(4.6,.5){\line(1,0){.2}}
\put(5,.5){\circle*{.2}}
\put(5,.5){\line(1,0){1}}
\put(6,.5){\circle*{.2}}
\put(6,.45){\line(1,0){1}}
\put(6,.55){\line(1,0){1}}
\put(7,.5){\circle*{.2}}
\put(3,.5){\circle{.35}}
\put(5,.5){\circle{.35}}
\put(6,.5){\circle{.35}}
\put(7,.5){\circle{.35}}
\put(1.8,.8){$e_{l}$}
\put(2.8,.8){$e_{l+1}$}
\put(6.8,.8){$e_n$}
\put(6.4,.4){$<$}
\put(1,.3){$\underbrace{\hphantom{1111111}}_{A_{l}}$}
\put(2.8,.3){$\underbrace{\hphantom{1111111111111111111111111}}_{C_{n-l}}$}
\end{picture}
\end{center}

For the commodity of the reader, we explicit the action of $\germ h $ and of $X_{\pm e_{i}}$ on $M(a)$ for both cases:
$$\left\{\begin{array}{lcl}
H_{e_{i}}\cdot x(k) & = & (k_{i}-k_{i+1})x(k), \: i=1,\ldots ,l-1\\
H_{e_{l}}\cdot x(k) & = & (-1-a_{l+1}+k_{l}-k_{l+1})x(k)\\
H_{e_{i}}\cdot x(k) & = & (a_{i}-a_{i+1}+k_{i}-k_{i+1})x(k), \: i=l+1,\ldots,n-1\\
H_{e_{n}}\cdot x(k) & = & (a_{n}+k_{n}+\frac{1}{2})x(k)\\
X_{e_{i}}\cdot x(k) & = & k_{i}x(k-\epsilon_{i+1}+\epsilon_{i}), \: i=1,\ldots ,l-1\\
X_{e_{l}}\cdot x(k) & = & k_{l}(a_{l+1}+k_{l+1})x(k-\epsilon_{l+1}+\epsilon_{l})\\
X_{e_{i}}\cdot x(k) & = & (a_{i+1}+k_{i+1})x(k-\epsilon_{i+1}+\epsilon_{i}), \: i=l+1,\ldots,n-1\\
X_{e_{n}}\cdot x(k) & = & \frac{1}{2}x(k+2\epsilon_{n})\\
X_{-e_{i}}\cdot x(k) & = & k_{i+1}x(k-\epsilon_{i}+\epsilon_{i+1}), \: i=1,\ldots l-1\\
X_{-e_{l}}\cdot x(k) & = & x(k-\epsilon_{l}+\epsilon_{l+1})\\
X_{-e_{i}}\cdot x(k) & = & (a_{i}+k_{i})x(k-\epsilon_{i}+\epsilon_{i+1}), \: i=l+1,\ldots,n-1\\
X_{-e_{n}}\cdot x(k) & = & -\frac{1}{2}(a_{n}+k_{n})(a_{n}+k_{n}-1)x(k-2\epsilon_{n})
\end{array}\right.$$

Now we claim the following:
\begin{theo}\label{thmMO}
The module $M(a)$ is a simple object of the category $\cal O _{\Phi,\theta_{a}}$. Moreover, the highest weight vectors for the action of $\germ l _{\theta_{a}}$ are the linear combinations of the $x(0,\ldots ,0,k_1,\ldots ,k_m)$ where $k_i \in \Z$ are such that $\sum_i  \: k_i\in 2\Z$.
\end{theo}
\begin{proof}
The proof goes along the same line as the proof of theorem \ref{thmNO}.

\end{proof}

\begin{rema}
Here again, the module $M(a)$ is in general on object of several categories $\cal O _{S,\theta}$.
\end{rema}

\subsubsection{Degree $1$ modules and cuspidality}\label{sssec:deg1cusp}

We shall now give one more property for the modules $N(a)$ and $M(a)$. We continue with the notations above. Note first that the action of $\germ l _{\Phi\setminus\theta_{a}}$ stabilizes the vector space consisting of all the $\germ l _{\theta_{a}}$-highest weight vectors. Thus this vector space has a structure of $\germ l _{\Phi\setminus\theta_{a}}$--module, which is cuspidal and one can also show it is simple by using the explicit action of $\germ l _{\Phi\setminus\theta_{a}}$. In fact, we do better:

\begin{prop}\label{propNM}
The modules $N(a)$ and $M(a)$ splits into a direct sum of simple cuspidal $\germ l _{\Phi\setminus\theta_{a}}$-modules.
\end{prop}
\begin{proof} Let us prove the proposition for $N(a)$. From theorem \ref{thmNO} we already know that the action of $\germ l _{\Phi\setminus\theta_{a}}$ on $N(a)$ is cuspidal. Let $x(k)\in N(a)$. Consider the $\germ l _{\Phi\setminus\theta_{a}}$--module $V(k)$ generated by $x(k)$. Let $X\in {\germ l _{\Phi\setminus\theta_{a}}}$ be a weight vector of weight $\alpha$. Then $X\cdot x(k)$ is again a weight vector in $V(k)$ which is non zero since the action of $X$ is cuspidal. On the other hand, if $Y \in {\germ l _{\Phi\setminus\theta_{a}}}$ is a vector of weight $-\alpha$, then $Y\cdot (X\cdot x(k))$ is a non zero vector (since the action of $Y$ is injective) having the same weight as $x(k)$. As $N(a)$ is a degree $1$ module, $Y\cdot (X\cdot x(k))$ should then be a non zero scalar multiple of $x(k)$. This proves that $V(k)$ is simple. But $N(a)$ is generated as a vector space by the various $x(k)$. Thus the proposition is proved. The proof is of course the same for $M(a)$.

\end{proof}


\section{Classification of the simple modules in $\mathcal O _{\Phi,\theta}$}\label{sec:GOirr}


In this part, {\bf we assume that $\germ g $ is a simple Lie algebra}. We fix a Cartan subalgebra $\germ h $ and denote by $\cal R $ the corresponding root system. We also fix a basis $\Phi$ of simple roots of $\cal R $. The aim of this section is the study of the various categories $\cal O _{\Phi,\theta}$ where $\theta \subset \Phi$. Note that if $\theta=\Phi$, then this category reduces to the semi-simple category whose objects are the direct sum of simple highest weight $\germ g $-modules. On the other hand, if $\theta=\emptyset$, then we get the category of cuspidal modules. Therefore, {\bf in what follows we shall always assume that $\emptyset\not=\theta\not=\Phi$}.

Let $L$ be a simple module in $\cal O _{\Phi,\theta}$. Then using Fernando's theorem \ref{thmFe}, we see that $L\cong L(\germ p _{\Phi\setminus \theta},C)$ where $C$ is a simple cuspidal $\germ l _{\Phi\setminus\theta}$-module. Thus to understand the simple module in $\cal O _{\Phi,\theta}$ it suffices to know which of the above modules $L(\germ p _{\Phi\setminus \theta},C)$ satisfy the restriction condition of the category $\cal O _{\Phi,\theta}$. In the sequel, we shall write $\germ p $ instead of $\germ p _{\Phi\setminus\theta}$, $\germ l $ instead of $\germ l _{\Phi\setminus\theta}$, $\germ n $ instead of $\germ n ^{+}_{\Phi\setminus\theta}$, and $L(C)$ instead of $L(\germ p _{\Phi\setminus \theta},C)$. We shall also need to consider the generalized Verma module $V(C):=V(\germ p _{\Phi\setminus \theta},C)$. We shall denote by $p:V(C)\rightarrow L(C)$ the natural projection and by $K(C)$ the kernel of this projection.

Before going further, let us state the main results we are going to prove. As we already mention, we want to find the conditions that the $\germ l $-module $C$ must fulfill in order that $L(C)$ be in category $\cal O _{\Phi,\theta}$, that is in order that $L(C)$ satisfies the restriction condition. We shall prove the following results:

\begin{theo}\label{thm_{A}}
Let $L(C)$ be a simple module in $\cal O _{\Phi,\theta}$. Then $C$ is a simple cuspidal $\germ l $-module of degree $1$.
\end{theo}

\begin{theo}\label{thm_{B}}
Let $\germ g $ be a simple Lie algebra not of type $C$. Assume $L(C)$ is a simple module in $\cal O _{\Phi,\theta}$. Then the semisimple part of the algebra $\germ l $ is a sum of ideals of type $A$.
\end{theo}

\begin{theo}\label{thm_{C}}
Let $L(C)$ be a simple module in $\cal O _{\Phi,\theta}$. Then the semisimple part of the algebra $\germ l $ is simple of type $A$ or $C$.
\end{theo}

\begin{theo}\label{thmpcpl0}
Let $\germ g $ be a simple Lie algebra. Let $\theta \subset \Phi$ with $\theta\not=\Phi$ and $\theta\not=\emptyset$. Assume the pair $(\germ g , \Phi\setminus\theta)$ does not belong to table \ref{tab2}.Then we have:
\begin{enumerate}
\item There exists a non trivial module in $\cal O _{\Phi,\theta}$ if and only if $\germ g $ is isomorphic to $A_n$ and $\germ l '_{\Phi\setminus \theta}$ to $A_m$ with $m<n$ or if $\germ g $ is isomorphic to $C_n$ and $\germ l '_{\Phi\setminus\theta}$ is isomorphic to either the subalgebra $\germ {sl} _2$ generated by the unique long simple root or to $C_k$.
\item For these pairs $(\Phi,\theta)$, the simple modules in $\cal O _{\Phi,\theta}$ are modules of degree $1$ except when $\germ g $ is of type $A_n$ for $n>2$ and $\germ l '_{\Phi\setminus\theta}$ is isomorphic to the $\germ sl _{2}$-algebra generated by one of the extreme simple roots.
\end{enumerate}
\end{theo}


\subsection{Proof of Theorems $A$, $B$, and $C$}\label{ssec:ABC}


\subsubsection{Proof of Theorem $A$}\label{sssec:A}

We now proceed to the proof of theorem $A$. This will require several lemmas. First of all, recall some facts about the action of $\germ g $ on $V(C)$. If $X\in \germ l $, then $X\cdot (1\otimes v)=1\otimes (X\cdot v)$ for any $v\in C$. If $X\in \germ n ^+$ then $X\cdot (1\otimes v)=0$ for any $v\in C$. More generally, for $X\in \germ n ^+$, we have $X\cdot (w\otimes v)=(ad(X)(w))\otimes v$ for any $v\in C$ and any $w\in\cal U (\germ g )$. Finally remark that $\germ l _{\theta}^+\subset\germ n ^+$. We will use these facts throughout this part without any further comments.

\begin{lemm}\label{lem2}
Let $L(C)$ be a simple module in $\cal O _{\Phi,\theta}$. Let $v\in C$ be a weight vector. Then the vector $p(1\otimes v)\in L(C)$ is a non zero weight vector and generates a simple highest weight $\germ l _{\theta}$--module.
\end{lemm}
\begin{proof} According to proposition \ref{GVM}, $p$ is an isomorphism from $1\otimes C$ onto its image. Thus, $p(1\otimes v)\not=0$. This vector is obviously a weight vector. Moreover, we have $\germ l _{\theta}^+\subset \germ n ^+$. Therefore, the $\germ l _{\theta}$--module generated by $p(1\otimes v)$ in $L(C)$ is a highest weight module. As such, it is indecomposable. On the other hand, the $\germ l _{\theta}$--module $L(C)$ is semisimple by the restriction condition of the category $\cal O _{\Phi,\theta}$. So the $\germ l _{\theta}$--module generated by $p(1\otimes v)$ should be semisimple too. But we have seen that it is indecomposable. Hence it must be simple, as asserted.

\end{proof}

\begin{lemm}\label{lem7}
Let $L(C)$ be a simple module in $\cal O _{\Phi,\theta}$. Let $\alpha\in \langle \Phi\setminus\theta\rangle ^+$. Let $\mbox{\boldmath $\beta$}=(\beta_1,\ldots ,\beta_i) \in (\langle \theta \rangle ^+)^i$ such that $\alpha+\beta_1+\cdots +\beta_k \in \cal R $ for any $k\leq i$. Let $v \in C$ be a weight vector. Then $p(X_{-(\alpha+\beta_1+\cdots +\beta_i)}\otimes v)\not= 0$ and we have: 
$$p(X_{-(\alpha+\beta_1+\cdots +\beta_i)}\otimes v) \in \cal U (\germ l _{\theta}^-)_{-(\beta_1+\cdots +\beta_i)}\cdot p(1 \otimes X_{-\alpha}v).$$
In particular, if $i=1$ and $\beta_1$ is a simple root, then there exists $\eta(v)\in \C$ non zero such that 
$$p(X_{-\alpha-\beta_1}\otimes v)=\eta(v)X_{-\beta_1}\cdot p(1\otimes X_{-\alpha}v).$$
\end{lemm}
\begin{proof} Set $w:=X_{\beta_1}\cdots X_{\beta_i}\in \cal U (\germ l ^+_{\theta})$. The adjoint action of $w$ on $X_{-(\alpha+\beta_1+\cdots +\beta_i)}$ gives a non zero multiple of $X_{-\alpha}$ (we can of course express explicitly this multiple by means of structure constants).  Thus the action of $w$ on $X_{-(\alpha+\beta_1+\cdots +\beta_i)}\otimes v \in V(C)$ gives a non zero multiple of $1\otimes X_{-\alpha}v$.

Now the cuspidality condition for $-\alpha \in \langle \Phi\setminus\theta\rangle$ ensures that $X_{-\alpha}v\not= 0$. Using proposition \ref{GVM}, this implies that $p(X_{-(\alpha+\beta_1+\cdots +\beta_i)}\otimes v)\not= 0$. On the other hand, we have seen that 
$w\cdot p(X_{-(\alpha+\beta_1+\cdots +\beta_i)}\otimes v)$ is a non zero multiple of $p(1\otimes X_{-\alpha}v).$ According to lemma \ref{lem2}, $p(1\otimes X_{-\alpha}v)$ generates a simple highest weight $\germ l _{\theta}$--module. Now the restriction condition for $L(C)$ implies that $\cal U (\germ l _{\theta})\cdot p(X_{-(\alpha+\beta_1+\cdots +\beta_i)}\otimes v)$ is semisimple. As it is generated by one element and should contain the simple module $\cal U (\germ l _{\theta})p(1\otimes X_{-\alpha}v)$, then it as to be simple and equal to this latter. By comparing the weights we deduce from this fact that 
$$p(X_{-(\alpha+\beta_1+\cdots +\beta_i)}\otimes v) \in \cal U (\germ l _{\theta}^-)_{-(\beta_1+\cdots +\beta_i)}\cdot p(1 \otimes X_{-\alpha}v).$$ If $i=1$ and $\beta_1$ is a simple root, then $\cal U (\germ l _{\theta}^-)_{-\beta_1}=\C X_{-\beta_1}$. This completes the proof of the lemma.

\end{proof}

\begin{lemm}\label{lem3}
Let $L(C)$ be a simple module in $\cal O _{\Phi,\theta}$. Let $\alpha \in \langle \Phi\setminus\theta\rangle ^+$ be such that there exists $\beta \in \theta$ with $\alpha+\beta\in \cal R $. Let $v \in C$ be a weight vector. Then $X_{\alpha}X_{-\alpha}v\in \C v$.
\end{lemm}
\begin{proof} Consider $u:=X_{-\alpha-\beta}\otimes v \: \in V(C)$. From the previous lemma applied to $\alpha$ and $\mbox{\boldmath $\beta$}=\beta$, there is a non zero complex number $\eta$ such that 
\begin{align}\label{lem3eq2}
p(u)= & \eta p(X_{-\beta}\otimes X_{-\alpha}v).
\end{align}
Apply then $X_{\alpha+\beta} \in \germ n ^+$ to equation \eqref{lem3eq2}. We get:
\begin{align*}
p(H_{\alpha+\beta}\otimes v) = & \eta p([X_{\alpha+\beta},X_{-\beta}]\otimes X_{-\alpha}v).
\end{align*}
Let $\lambda \in \germ h ^*$ denote the weight of $v$. Let $c'$ denote the non zero structure constant such that $[X_{\alpha+\beta},X_{-\beta}]=c'X_{\alpha}\in \germ l $. Then the above equation becomes:
\begin{align*}
\lambda (H_{\alpha+\beta})p(1\otimes v)= & \eta c' p(1\otimes X_{\alpha}X_{-\alpha}v).
\end{align*}
Since $\eta $ and $c'$ are non zero, we get $$p(1 \otimes (\kappa v)-1\otimes (X_{\alpha}X_{-\alpha}v))=0,$$ with $\kappa = \frac{\lambda(H_{\alpha+\beta})}{\eta c'}$. As $p$ is an isomorphism from $1\otimes C$ onto $C$, we deduce that
$$1 \otimes (\kappa v)-1\otimes (X_{\alpha}X_{-\alpha}v)=0$$ and therefore that $$\kappa v-X_{\alpha}X_{-\alpha}v=0.$$ Hence $X_{\alpha}X_{-\alpha}v=\kappa v\in \C v$ as asserted.

\end{proof}

\begin{lemm}\label{lem4}
Let $N$ be a weight $\germ g $--module. Let $\alpha, \gamma \in \cal R $ be such that
\begin{enumerate}
\item $\alpha+\gamma \in \cal R $ and $\alpha-\gamma \not\in \cal R $.
\item $X_{\alpha}X_{-\alpha}v\in \C v$ and $X_{\gamma}X_{-\gamma}v\in \C v$, for any weight vector $v \in N$.
\end{enumerate}
Then $X_{\alpha+\gamma}X_{-\alpha-\gamma}v\in \C v$,  for any weight vector $v \in N$.
\end{lemm}
\begin{proof} There are two non zero structure constants $c$ and $d$ such that $cX_{\alpha+\gamma}=[X_{\alpha},X_{\gamma}]$ and $dX_{-\alpha-\gamma}=[X_{-\alpha},X_{-\gamma}]$. Thus, in the universal enveloping algebra we get:
$$cdX_{\alpha+\gamma}X_{-\alpha-\gamma}=(X_{\alpha}X_{\gamma}-X_{\gamma}X_{\alpha})(X_{-\alpha}X_{-\gamma}-X_{-\gamma}X_{-\alpha}).$$ Let us develop this expression. Since $\alpha-\gamma \not\in \cal R $ by our hypothesis, the vectors $X_{\alpha}$ and $X_{-\gamma}$ commute as well as $X_{-\alpha}$ and $X_{\gamma}$. Thus, we obtain
\begin{align*}
cdX_{\alpha+\gamma}X_{-\alpha-\gamma}= & X_{\alpha}X_{-\alpha}X_{\gamma}X_{-\gamma}-X_{\alpha}X_{\gamma}X_{-\gamma}X_{-\alpha}\\
 & -X_{\gamma}X_{\alpha}X_{-\alpha}X_{-\gamma}+X_{\gamma}X_{-\gamma}X_{\alpha}X_{-\alpha}.
\end{align*}
Let us apply this expression to the weight vector $v$. We find:
\begin{align*}
cdX_{\alpha+\gamma}X_{-\alpha-\gamma}\cdot v= & (X_{\alpha}X_{-\alpha})(X_{\gamma}X_{-\gamma}\cdot v)-X_{\alpha}(X_{\gamma}X_{-\gamma})(X_{-\alpha}\cdot v)\\
 & -X_{\gamma}(X_{\alpha}X_{-\alpha})(X_{-\gamma}\cdot v)+(X_{\gamma}X_{-\gamma})(X_{\alpha}X_{-\alpha}\cdot v).
\end{align*}
Thanks to our second hypothesis we must have
$$X_{\alpha}X_{-\alpha}(X_{-\gamma}v) \in  \C X_{-\gamma}v \mbox{ and } X_{\gamma}X_{-\gamma}(X_{-\alpha}v) \in  \C X_{-\alpha}v.$$ From this, we deduce the lemma.

\end{proof}

\begin{lemm}\label{lem5}
Let $N$ be a weight $\germ g $--module. Let $\alpha, \gamma \in \cal R $ be such that
\begin{enumerate}
\item $\alpha+\gamma \in \cal R $, $2\alpha+\gamma \in \cal R $ and $\alpha-\gamma \not\in \cal R $.
\item $X_{\alpha}X_{-\alpha}v\in \C v$ and $X_{\gamma}X_{-\gamma}v\in \C v$, for any weight vector $v \in N$.
\end{enumerate}
Then $X_{2\alpha+\gamma}X_{-2\alpha-\gamma}v\in \C v$, for any weight vector $v\in N$.
\end{lemm}
\begin{proof} There is a non zero structure constant $c$ such that $cX_{2\alpha+\gamma}=[X_{\alpha},[X_{\alpha},X_{\gamma}]]$. The proof is now analogous to the previous one.

\end{proof}

\begin{lemm}\label{lem5b}
Let $N$ be a weight $\germ g $--module. Let $\alpha, \gamma \in \cal R $ be such that
\begin{enumerate}
\item $\alpha+\gamma \in \cal R $, $2\alpha+\gamma \in \cal R $, $3\alpha+\gamma\in \cal R $ and $\alpha-\gamma \not\in \cal R $.
\item $X_{\alpha}X_{-\alpha}v\in \C v$ and $X_{\gamma}X_{-\gamma}v\in \C v$, for any weight vector $v \in N$.
\end{enumerate}
Then $X_{3\alpha+\gamma}X_{-3\alpha-\gamma}v\in \C v$, for any weight vector $v\in N$.
\end{lemm}
\begin{proof} There is a non zero structure constant $c$ such that $cX_{3\alpha+\gamma}=[X_{\alpha},[X_{\alpha},[X_{\alpha},X_{\gamma}]]]$. The proof is now analogous to the previous one.

\end{proof}

\begin{lemm}\label{lem6}
Let $N$ be a simple weight $\germ g $--module. Assume that for any $\alpha \in \cal R $ and any weight vector $v \in N$, we have $X_{\alpha}X_{-\alpha}v\in \C v$. Then $N$ is a module of degree $1$.
\end{lemm}
\begin{proof} Let $v\in N$ be a weight vector. We shall prove that $\cal U (\germ g )_0 v\subset \C v$ (where $\cal U (\germ g )_0$ is the commutant of $\germ h $ in $\cal U (\germ g )$). Since $v$ is a weight vector we have by definition $\cal U (\germ h )v\subset \C v$. But we know that the algebra $\cal U (\germ g )_0$ is generated by $\cal U (\germ h )$ and some monomials of the form $u=X_1\cdots X_k$ with $k \in \N$ and $X_i\in \germ g _{\pm \beta}$ for some simple root $\beta \in \cal R $. Such a monomial belongs to $\cal U (\germ g )_0$ if and only if the multiplicity of each simple root $\beta$ in $u$ is equal to that of $-\beta$. Note in particular that the integer $k$ should then be even. 

Let us show that $u\cdot v\in \C v$ by induction on $k$. For $k=0$, we have $u=1$ and so $u\cdot v=v$. For $k=2$, we have either $u=X_{\beta}X_{-\beta}$ or $u=X_{-\beta}X_{\beta}$ for some simple root $\beta$. In the first case, we have $u\cdot v \in \C v$ by our hypothesis. In the second case, we notice that $u=X_{\beta}X_{-\beta}-H_{\beta}$. Thus we also get here that $u\cdot v\in \C v$.

Assume then that $u'\cdot v \in \C v$ for any monomial $u'$ of degree less than $k$ and any weight vector $v$. Let $u=X_1\cdots X_k$ be a monomial of degree $k$. Note that for any  $i$, $X_i\cdots X_k\cdot v $ is again a weight vector. Therefore, if $u$ contains a submonomial $X_{j}\cdots X_{i-1}$ belonging  to $\cal U (\germ g )_0$ then our induction hypothesis implies that 
$$X_1\cdots X_{j-1}(X_j\cdots X_{i-1})(X_i \cdots X_k\cdot v)\in \C X_1\cdots X_{j-1}(X_i\cdots X_k\cdot v).$$ Since $u\in \cal U (\germ g )_0$ then $X_1\cdots X_{j-1}X_i\cdots X_k\in \cal U (\germ g )_0$ and we can apply once again our induction hypothesis to deduce that $u\cdot v\in \C v$.

Thus it suffices to show that $u$ does contain a submonomial belonging to $\cal U (\germ g )_0$. Assume it is not the case. Without lake of generality we can suppose that $X_1\in \germ g _{+\beta}$ for some simple root $\beta$. Let $i_1$ be the first integer greater than $1$ such that $X_{i_1}$ belongs to a root space associated to a simple root. For any integer $1<j<i_{1}$, the vector $X_{j}$ commutes with $X_{i_{1}}$ except if the weight of $X_{j}$ is the opposite of that of $X_{i_{1}}$. But if such a vector occurs then we would have a submonomial (of degree $2$) of $u$ belonging to $\cal U (\germ g )_0$ contrary to our assumption. Thus we can suppose that $i_{1}=2$.

We then look at $i_{2}$, the first integer greater than $2$ such that $X_{i_2}$ belongs to a root space associated to a simple root. The same reasoning shows that we can suppose that $i_{2}=3$. From this kind of reasoning we deduce that we can suppose that the first $k/2$ vectors belong to root spaces associated to simple roots.  Let $\beta$ be the simple root such that $X_{k/2}\in \germ g _{\beta}$. Necessarily, the last $k/2$  vectors belong to root spaces associated with negative roots. Moreover, among these vectors there is at least one belonging to $\germ g _{-\beta}$. Let $i$ be the smallest integer such that $X_i\in \germ g _{-\beta}$. Then for any $k/2<j<i$, $X_j$ commutes with $X_{k/2}$. Therefore we can find in $u$ a submonomial, $X_{k/2}X_i$, belonging to $\cal U (\germ g )_0$, contrary to our assumption. This proves that $u$ always contain a submonomial belonging to $\cal U (\germ g )_0$. Hence we have shown that $u\cdot v \in \C v$.

So we have $\cal U (\germ g )_0 v\subset \C v$. Lemire's correspondence \cite{Le68} gives then the lemma.

\end{proof}

\begin{proof} (theorem \ref{thm_{A}}) Thanks to lemma \ref{lem6}, it suffices to prove that for any $\alpha \in \langle \Phi\setminus\theta\rangle $ and any weight vector $v \in C$, we have$X_{\alpha}X_{-\alpha}v\in \C v$. Since $X_{\alpha}X_{-\alpha}-X_{-\alpha}X_{\alpha}\in \germ h $ for any $\alpha$, it suffices to prove it only for positive $\alpha $.

Let us fix some weight vector $v \in C$. Let $\alpha \in \langle \Phi\setminus\theta\rangle ^+$. If there is $\beta \in \theta$ such that $\alpha+\beta\in \cal R $, then lemma \ref{lem3} applied to $\beta$ and $\alpha$ gives $X_{\alpha}X_{-\alpha}v\in \C v$. Otherwise, let $\alpha'\in \langle \Phi\setminus\theta\rangle ^+$ be such that
\begin{itemize}
\item[$\bullet$] $\alpha+\alpha'\in \cal R $ and $\alpha-\alpha'\not\in \cal R $,
\item[$\bullet$] $\exists \: \beta \in \theta, \: \beta+\alpha'\in \cal R \mbox{ and } \beta+\alpha'+\alpha\in \cal R $.
\end{itemize}
Such a root $\alpha'$ does exist since the sets $\Phi\setminus\theta$ and $\theta$ form a partition of the Dynkin diagram of $(\germ g ,\germ h )$ which is connected. Lemma \ref{lem3} applied to $\beta$ and $\alpha'$ on one hand and to $\beta $ and $\alpha'+\alpha$ on the other hand gives $X_{\alpha'}X_{-\alpha'}v\in \C v $ and $X_{\alpha+\alpha'}X_{-\alpha-\alpha'}v\in \C v$.

Now if $2\alpha'+\alpha\not\in \cal R $, lemma \ref{lem4} applied to $-\alpha'$ and $\alpha'+\alpha$ gives $X_{\alpha}X_{-\alpha}v \in \C v$. If $2\alpha'+\alpha\in \cal R $ and $3\alpha'+\alpha\not\in \cal R $, then notice that $\beta+2\alpha'+\alpha \in \cal R $. So we can apply lemma \ref{lem3} to $\beta$ and $2\alpha'+\alpha$ to get $X_{2\alpha'+\alpha}X_{-(2\alpha'+\alpha)}v\in \C v$ and then lemma \ref{lem5} to the roots $-\alpha'$ and $2\alpha'+\alpha$ to obtain $X_{\alpha}X_{-\alpha}v\in \C$. If $2\alpha'+\alpha\in \cal R $ and $3\alpha'+\alpha\in \cal R $, then $\beta+2\alpha'+\alpha\in \cal R $ and $\beta+3\alpha'+\alpha\in \cal R $. Therefore we apply lemma \ref{lem3} to $\beta $ and $2\alpha'+\alpha\in \cal R $ and to $\beta$ and $3\alpha'+\alpha\in \cal R $ to get $X_{2\alpha'+\alpha}X_{-(2\alpha'+\alpha)}v\in \C v $ and $X_{3\alpha'+\alpha}X_{-(3\alpha'+\alpha)}v\in \C v $. Now $4\alpha'+\alpha\not\in \cal R $ (see for instance  \cite[table 1 p.45]{Hu78}). Thus lemma \ref{lem5b} applied to $-\alpha'$ and $3\alpha'+\alpha$ implies $X_{\alpha}X_{-\alpha}v\in \C v$. Hence we show that we always have $X_{\alpha}X_{-\alpha}v\in \C v$. The theorem is thus proved.

\end{proof}

\subsubsection{Proof of theorem $B$}\label{sssec:B}

We now proceed to the proof of theorem $B$:

\begin{theo}\label{thmB}
Assume $\germ g $ is not of type $C$. Let $L(C)$ be a simple module in $\cal O _{\Phi,\theta}$, then the semisimple part of the algebra $\germ l $ is a sum of ideals of type $A$.
\end{theo}
\begin{proof} Thanks to theorem \ref{thmFeC}, it suffices to show that $\germ l '$ cannot have an ideal of type $C$. If $\germ l '$ does contain an ideal of type $C$ then $\germ g $ is of type $B_n$ (for $n\geq 3$) or $F_4$ and the Dynkin diagram of $\germ g $ contains the following piece:

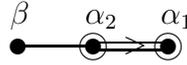
\setlength{\unitlength}{1cm}
\begin{center}
\begin{picture}(4,1.2)
\put(1,.5){\circle*{.2}}
\put(1,.5){\line(1,0){1}}
\put(2,.5){\circle*{.2}}
\put(2,0.45){\line(1,0){1}}
\put(2,.55){\line(1,0){1}}
\put(3,.5){\circle*{.2}}
\put(2,.5){\circle{.35}}
\put(3,.5){\circle{.35}}
\put(.9,.8){$\beta$}
\put(1.9,.8){$\alpha_2$}
\put(2.9,.8){$\alpha_1$}
\put(2.4,.4){$>$}
\end{picture}
\end{center} 

Let then $v \in C$ be a weight vector. Consider $u:=p(X_{-\alpha_2-\beta}\otimes v)\in L(C)$. As $\beta$ is a simple root, lemma \ref{lem7} implies that there is a non zero complex number $\eta(v)$ such that $u=\eta(v) p(X_{-\beta}\otimes X_{-\alpha_2}v)$. Apply the vector $X_{\beta+\alpha_2+2\alpha_1}$ to this equality. We get:
\begin{align*}
p([X_{\beta+\alpha_2+2\alpha_1},X_{-\alpha_2-\beta}]\otimes v)= & \eta(v) \times p([X_{\beta+\alpha_2+2\alpha_1},X_{-\beta}] \otimes X_{-\alpha_2}v).
\end{align*}
But $[X_{\beta+\alpha_2+2\alpha_1},X_{-\alpha_2-\beta}]=0$. Moreover, there exists a non zero structure constant $c$ such that $[X_{\beta+\alpha_2+2\alpha_1},X_{-\beta}]=cX_{\alpha_2+2\alpha_1}\in \germ l $. Therefore: 
\begin{align*}
0 = & \eta (v)c p(1\otimes X_{\alpha_2+2\alpha_1}X_{-\alpha_2}v).
\end{align*}
Now the cuspidality condition for $L(C)$ implies that $X_{\alpha_2+2\alpha_1}X_{-\alpha_2}v\not= 0$ and thus $p(1\otimes X_{\alpha_2+2\alpha_1}X_{-\alpha_2}v)\not= 0$. This contradicts $\eta\not=0$ and the proof is completed.

\end{proof}

\subsubsection{Proof of theorem $C$}\label{sssec:C}

Now we turn to the proof of theorem $C$:

\begin{theo}\label{thmS}
Let $L(C)$ be a simple module in $\cal O _{\Phi,\theta}$. Then the semisimple part of the algebra $\germ l $ is simple, of type $A$ or $C$.
\end{theo}
\begin{proof} Assume this is not the case. For simplicity, we suppose then that $\germ l '$ is a sum of two simple ideals of type $A$ or $C$. We shall denote these ideals by $\germ l _1$ and $\germ l _2$. Set $S_i$ for the root basis of $(\germ l _i,\germ h \cap\germ l _{i})$ deduced from $\Phi\setminus\theta$.

Let $v\in C$ be a weight vector. Let $\alpha \in S_1$, $\alpha '\in S_2$ and $\beta_1,\ldots ,\beta_k \in \theta$ such that $\alpha+\beta_1+\cdots +\beta_k+\alpha'\in \cal R $. We will suppose that the simple roots $\beta_i$ are all distinct. Consider $u:=X_{-(\alpha+\beta_1+\cdots \beta_k)}\otimes v\in V(C)$. Lemma \ref{lem7} implies that $p(u)\not= 0$ and that
$$p(u)\in \cal U (\germ l _{\theta}^-)_{-(\beta_1+\cdots +\beta_k)}p(1\otimes X_{-\alpha}v).$$ But the adjoint action of $X_{\alpha+\beta_1+\cdots +\beta_k+\alpha'}$ on $\cal U (\germ l _{\theta}^-)_{-(\beta_1+\cdots +\beta_k)}$ is trivial (since it is trivial on every vector of the form $X_{-(\beta_i+\cdots +\beta_j)}$ for $1\leq i\leq j\leq k$). Thus the action of $X_{\alpha+\beta_1+\cdots +\beta_k+\alpha'}$ on $p(u)$ must be trivial too. This action is given by:
\begin{align*}
X_{\alpha+\beta_1+\cdots +\beta_k+\alpha'}\cdot p(u) = & p([X_{\alpha+\beta_1+\cdots +\beta_k+\alpha'},X_{-(\alpha+\beta_1+\cdots \beta_k)}]\otimes v).
\end{align*}
There is a non zero structure constant $c$ such that
$$[X_{\alpha+\beta_1+\cdots +\beta_k+\alpha'},X_{-(\alpha+\beta_1+\cdots \beta_k)}]=cX_{\alpha'}\in \germ l .$$ We get then
\begin{align*}
X_{\alpha+\beta_1+\cdots +\beta_k+\alpha'}\cdot p(u) = & cp(1\otimes X_{\alpha'}v).
\end{align*}
The cuspidality condition for $L(C)$ ensures that $X_{\alpha'}v\not= 0$ and so $p(1\otimes X_{\alpha'}v)\not= 0$. This is a contradiction with $X_{\alpha+\beta_1+\cdots +\beta_k+\alpha'}\cdot p(u)=0$.

\end{proof}

\subsubsection{A first reduction}\label{sssec:1red}

We end this section by showing that for some $(\Phi,\theta)$ the simple module $L(C)$ cannot be in $\cal O _{\Phi,\theta}$. This will use lemma \ref{lem7} and the possibility of considering large positive roots.

\begin{table}[htbp]
\begin{center}
\begin{tabular}{|c|c|}
\hline 
\textbf{Type} & $\mathbf{\Phi\setminus\theta}$\\
\hline
$B_n$ ($n>3$) & $\{e_i\}, \: i\not=1$\\
\hline
$B_n$ & $\{e_i,\ldots , e_{i+k}\}, \: i+k<n, \: i\geq 1, \: k>1$\\
\hline
$C_n$ & $\{e_i\}, \: i<n$\\
\hline
$C_n$ & $\{e_i,\ldots , e_{i+k}\}, \: i+k<n, k>1$\\
\hline
$F_4$ & $\{e_i\}$\\
\hline
$F_4$ & $\{e_1,e_2\}$ or $\{e_3,e_4\}$\\
\hline
$D_n$ & $A_k, \: k>1$\\
\hline
$D_n$ & $\{e_i\}, \: i\not\in \{1,n-1,n\}$\\
\hline
$E$ & $A_k,$ except $\{e_1\}$ or $\{e_6\}$ for $E_6$, and $\{e_7\}$ for $E_7$\\
\hline
$G_2$ & $\{e_1\}$ or $\{e_2\}$\\
\hline
\end{tabular}
\end{center}
\caption{A first reduction}\label{tab1}
\end{table}

\begin{theo}\label{thm1red}
Let $(\Phi,\theta)$ be in table \ref{tab1}. Let $L(C)$ be a simple module. Then $L(C)$ is not in category $\cal O _{\Phi,\theta}$. Consequently the category $\cal O _{\Phi,\theta}$ is trivial.
\end{theo}
\begin{proof} Assume on the contrary that $L(C)$ is in $\cal O _{\Phi,\theta}$. 
\begin{itemize}
\item[$\bullet$] Suppose we can find in the Dynkin diagram of $\germ g $ the following piece:

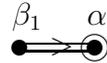
\setlength{\unitlength}{1cm}
\begin{center}
\begin{picture}(3,1.2)
\put(1,.5){\circle*{.2}}
\put(1,0.45){\line(1,0){1}}
\put(1,.55){\line(1,0){1}}
\put(2,.5){\circle*{.2}}
\put(2,.5){\circle{.35}}
\put(.9,.8){$\beta_1$}
\put(1.9,.8){$\alpha$}
\put(1.4,.4){$>$}
\end{picture}
\end{center} 

Let $v$ be a weight vector in $C$. Apply lemma \ref{lem7} to $\alpha$ and $\mbox{\boldmath $\beta$}=(\beta_1)$. Since $\beta_1$ is a simple root, there is a non zero complex number $\eta (v)$ such that $u=\eta(v)p(X_{-\beta_1}\otimes X_{-\alpha}v)$. Apply now the vector $X_{\beta_1+2\alpha}\in \germ n ^+$ to this equality. We get:
\begin{align*}
p([X_{\beta_1+2\alpha},X_{-\alpha-\beta_1}]\otimes v) = & \eta (v) p([X_{2\alpha+\beta_1},X_{-\beta_1}]\otimes X_{-\alpha}v).
\end{align*}
But $[X_{2\alpha+\beta_1},X_{-\beta_1}]=0$. So we have $p([X_{\beta_1+2\alpha},X_{-\alpha-\beta_1}]\otimes v)=0$. Moreover there is a non zero structure constant $c$ such that the following holds: $[X_{\beta_1+2\alpha},X_{-\alpha-\beta_1}]=cX_{\alpha}$. Thus we have $cp(1\otimes X_{\alpha}v)=0$. As $C$ is cuspidal, the action of $X_{\alpha}$ on $v$ is non zero. Therefore $p(1\otimes X_{\alpha}v)\not= 0$. This is a contradiction. From this reasoning the theorem is proved for the following cases in table \ref{tab1}: $(B_n,\{e_n\})$, $(C_n,\{e_{n-1},\ldots ,e_i\})$ with $i\leq n-1$, $(F_4,\{e_3\})$, $(F_4,\{e_3,e_4\})$.

\item[$\bullet$] Suppose we can find in the Dynkin diagram of $\germ g $ the following piece (with $k>0$):

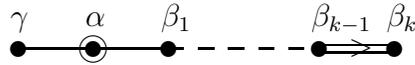
\setlength{\unitlength}{1cm}
\begin{center}
\begin{picture}(7,1.2)
\put(1,.5){\circle*{.2}}
\put(1,.5){\line(1,0){1}}
\put(2,.5){\circle*{.2}}
\put(3,.5){\line(1,0){.2}}
\put(3.4,.5){\line(1,0){.2}}
\put(3.8,.5){\line(1,0){.2}}
\put(4.2,.5){\line(1,0){.2}}
\put(4.6,.5){\line(1,0){.2}}
\put(3,.5){\circle*{.2}}
\put(2,.5){\line(1,0){1}}
\put(5,.5){\circle*{.2}}
\put(5,.45){\line(1,0){1}}
\put(5,.55){\line(1,0){1}}
\put(6,.5){\circle*{.2}}
\put(2,.5){\circle{.35}}
\put(.9,.8){$\gamma$}
\put(4.9,.8){$\beta_{k-1}$}
\put(5.9,.8){$\beta_k$}
\put(1.9,.8){$\alpha$}
\put(2.9,.8){$\beta_1$}
\put(5.4,.4){$>$}
\end{picture}
\end{center}

We use the same method as above, applying lemma \ref{lem7} to $\alpha$ and $\mbox{\boldmath $\beta$}=(\gamma,\beta_1+\ldots +\beta_k,\beta_1+\ldots +\beta_k)$. Set $u:=p(X_{-\gamma-\alpha-2\beta_1-\ldots -2\beta_k}\otimes v)$. Lemma \ref{lem7} gives then $u\not=0$ and $u\in p(\cal U (\germ l _{\theta}^-)_{-\gamma-2(\beta_1+\cdots +\beta_k)}\otimes X_{-\alpha}v)$. Let us apply to $u$ the vector $X_{\gamma+2\alpha+2\beta_1+\cdots +2\beta_k}\in \germ n ^+$. First note that the adjoint action of $X_{\gamma+2\alpha+2\beta_1+\cdots +2\beta_k}$ on $\cal U (\germ l _{\theta}^-)_{-\gamma-2(\beta_1+\cdots +\beta_k)}$ is trivial. Therefore we must have:
$$p([X_{\gamma+2\alpha+2\beta_1+\cdots +2\beta_k},X_{-(\gamma+\alpha+2\beta_1+\cdots +2\beta_k)}]\otimes v)=0.$$ But now there a non zero structure constant $c$ such that 
$$[X_{\gamma+2\alpha+2\beta_1+\cdots +2\beta_k},X_{-(\gamma+\alpha+2\beta_1+\cdots +2\beta_k)}]=c X_{\alpha}.$$ Thus we should have $cp(1\otimes X_{\alpha}v)=0$. The cuspidality of the module $C$ implies that $X_{\alpha}v\not= 0$ and so we have $p(1\otimes X_{\alpha}v)\not= 0$. This is a contradiction. From this reasoning the theorem is proved for the following cases in table \ref{tab1}: $(B_n,\{e_i\})$ pour $i<n$, $(F_4,\{e_2\})$, $(F_4,\{e_1\})$.

\item[$\bullet$] Suppose we can find in the Dynkin diagram of $\germ g $ the following piece (with $k>0$):

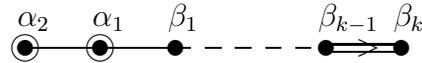
\setlength{\unitlength}{1cm}
\begin{center}
\begin{picture}(7,1.2)
\put(1,.5){\circle*{.2}}
\put(1,.5){\line(1,0){1}}
\put(2,.5){\circle*{.2}}
\put(3,.5){\line(1,0){.2}}
\put(3.4,.5){\line(1,0){.2}}
\put(3.8,.5){\line(1,0){.2}}
\put(4.2,.5){\line(1,0){.2}}
\put(4.6,.5){\line(1,0){.2}}
\put(3,.5){\circle*{.2}}
\put(2,.5){\line(1,0){1}}
\put(5,.5){\circle*{.2}}
\put(5,.45){\line(1,0){1}}
\put(5,.55){\line(1,0){1}}
\put(6,.5){\circle*{.2}}
\put(2,.5){\circle{.35}}
\put(1,.5){\circle{.35}}
\put(.9,.8){$\alpha_2$}
\put(4.9,.8){$\beta_{k-1}$}
\put(5.9,.8){$\beta_k$}
\put(1.9,.8){$\alpha_1$}
\put(2.9,.8){$\beta_1$}
\put(5.4,.4){$>$}
\end{picture}
\end{center}

We apply lemma \ref{lem7} to the vector $u=p(X_{-\alpha_1-2\beta_1-\ldots -2\beta_k}\otimes x(b))$ with $\alpha=\alpha_1$ and $\mbox{\boldmath $\beta$}=(\beta_1+\ldots +\beta_k,\beta_1+\ldots +\beta_k)$. We get $u\not=0$ and $u\in \cal U (\germ g )_{-2\beta_1-\ldots -2\beta_k}$. Apply the vector $X_{\alpha_{2}+2\alpha_1+2\beta_1+\cdots +2\beta_l}\in \germ n ^+$ to this. We then obtain the same contradiction as above.  From this reasoning the theorem is proved for the following cases in table \ref{tab1}: $(B_n,\{e_i,\ldots ,e_{i+k}\})$ ($i+k<n$) et $(F_4,\{e_1,e_2\})$.

\item[$\bullet$] All the remaining cases are proved with the same method, by applying lemma \ref{lem7} to a well chosen vector. We omit the details.
\end{itemize}

\end{proof}


\subsection{Type $A$ case}\label{ssec:typeA}



\subsubsection{Case $\germ l '=\germ sl _{2}$ (I)}\label{sec:A1An}


We consider here the following case: $\germ g =A_n$ and $\Phi\setminus\theta=\{e_1\}$ (or $\{e_n\}$) which corresponds to the Dynkin diagram:

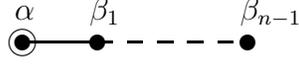
\setlength{\unitlength}{1cm}
\begin{center}
\begin{picture}(5,1.2)
\put(1,0.5){\line(1,0){1}}
\put(1,0.5){\circle*{.2}}
\put(2,0.5){\circle*{.2}}
\put(2,0.5){\line(1,0){.2}}
\put(2.4,0.5){\line(1,0){.2}}
\put(2.8,0.5){\line(1,0){.2}}
\put(3.2,0.5){\line(1,0){.2}}
\put(3.6,0.5){\line(1,0){.2}}
\put(4,0.5){\circle*{.2}}
\put(1,0.5){\circle{.35}}
\put(.9,0.8){$\alpha$}
\put(1.9,0.8){$\beta_1$}
\put(3.9,0.8){$\beta_{n-1}$}
\end{picture}
\end{center}

Let $L(C)$ be a simple module in the category $\cal O _{\Phi,\theta}$. In this case, $\germ l '=A_1$. Therefore the module $C$ is of the form $C=N(a_1,a_2)$ with $a_{1}$ and $a_{2}$ two non integer complex numbers, and the center of $\germ l $ acts as scalar operators on $C$. Set $A=a_1+a_2$. For the commodity of the reader, we recall that $C$ is generated by vectors $x(k)$ for $k\in \Z$, and the action of $\germ l '$ on $x(k)$ is given by the following recipe:
$$
\left\{ \begin{array}{ccc}
H_{\alpha}\cdot x(b) & = & (a_1-a_2+2k)x(k)\\
X_{\alpha}\cdot x(b) & = & (a_2-k)x(k+1)\\
X_{-\alpha}\cdot x(b) & = & (a_1+k)x(b-1)\\
\end{array}\right.$$

\begin{lemm}\label{lemA12}
We have $p(X_{-\beta_1-\alpha}\otimes x(k))=\frac{c+a_1+k}{a_2-k+1}p(X_{-\beta_1}\otimes x(k-1))$, where $c=0$ or $c=-1-A$.
\end{lemm}
\begin{proof} Set $u=p(X_{-\beta_1-\alpha}\otimes x(k))$. As $X_{-\alpha}\cdot x(b)=(a_1+k)x(k-1)$ and $a_1\not =0$, lemma \ref{lem7} applied to $\alpha$ and $\mbox{\boldmath $\beta$}=(\beta_1)$) ensures that there is a non zero complex number $\eta(k)$ such that 
\begin{eqnarray}\label{eq1}
u=\eta(k)p(X_{-\beta_1}\otimes x(k-1)).
\end{eqnarray}
On the other hand, $2H_{\beta_1}+H_{\alpha}$ is in the center of $\germ l $. Therefore it acts on $C$ by some constant. Let $c(k)$ denotes the action of $H_{\beta_1}$ on $x(k)$. Then we must have $2c(k)+(a_1-a_2+2k)=cte$. Thus we have $c(k)=c+a_2+k$ for some constant $c$. Apply $X_{\beta_1}$ and $X_{\beta_1+\alpha}$ to equation \eqref{eq1}. We obtain the following equations:
$$\left\{\begin{array}{ccc}
p([X_{\beta_1},X_{-\beta_1-\alpha}]\otimes x(k)) & = & \eta (k)p([X_{\beta_1},X_{-\beta_1}]\otimes x(k-1))\\
p([X_{\beta_1+\alpha},X_{-\beta_1-\alpha}]\otimes x(k)) & = & \eta(k)p([X_{\beta_1+\alpha},X_{-\beta_1}]\otimes x(k-1))\\
\end{array}\right.$$
Now we have the following structure constants:
$$[X_{\beta_1},X_{-\beta_1-\alpha}]=X_{-\alpha}, \: [X_{\beta_1+\alpha},X_{-\beta_1}]= X_{\alpha}.$$ Hence we get:
$$\left\{\begin{array}{ccc}
(a_1+k) p(1\otimes x(k-1)) & = & \eta (k)c(k-1)p(1\otimes x(k-1))\\
(c(k)+a_1-a_2+2k)p(1\otimes x(k)) & = & \eta(k)(a_2-k+1)p(1\otimes x(k))\\
\end{array}\right.$$
Since $p(1\otimes x(k-1))\not= 0$ and $p(1\otimes x(k))\not= 0$, we deduce that:
$$\left\{\begin{array}{ccc}
a_1+k  & = & \eta (k)c(k-1)\\
(c(k)+a_1-a_2+2k) & = & \eta(k)(a_2-k+1)\\
\end{array}\right.$$
The solution of this system is:
$$c=0 \mbox{ or } c=-1-A \mbox{ and } \eta(k)=\frac{c+a_1+k}{a_2-k+1}.$$

\end{proof}

The previous lemma together with theorem \ref{thmNO} allow us to state a classification result for $\germ g =A_2$ :

\begin{coro}\label{corA12}
Assume $\germ g =A_2$. Let $\theta\subset \Phi$ be such that $\Phi\setminus\theta=\{e_1\}$. Then $M$ is a simple module in $\cal O _{\Phi,\theta}$ if and only if $M$ is isomorphic to some $N(a'_1,a'_2,0)$ with $a'_1, a'_2\in \C\setminus \Z$.
\end{coro}
\begin{proof} Let $M$ be a simple module in $\cal O _{\Phi,\theta}$. As we already mentioned we have $M=L(C)$ for some cuspidal module $C=N(a_1,a_2)$. Now, using lemma \ref{lemA12} we see that the action of $\cal U (\germ g )_{0}$ on the vector $x(0)$ is the same as the action of $\cal U (\germ g )_{0}$ on the vector $x(0,0,0) \in N(a_{1},a_{2},0)$ if $c=0$ or on the vector $x(0)\in N(-1-a_{2},-1-a_{1},0)$ if $c=-1-A$. Therefore we conclude from Lemire's correspondence \cite{Le68} that these modules are isomorphic. Conversely theorem \ref{thmNO} ensures that these modules are objects in the category $\cal O _{\Phi,\theta}$.

\end{proof}

Unfortunately, our method is rather inefficient to treat the general case, which seems to be more complicate. As an example we treat the case of $A_{3}$ in appendix \ref{append}.


\subsubsection{Case $\germ l '=\germ sl _{2}$ (II)}\label{subsec:A1N}


Here we consider the case $\Phi\setminus\theta=\{e_l\}$, with $1<l<n$. Hence the Dynkin diagram of $\germ g =A_n$ contains the following piece:

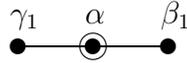
\setlength{\unitlength}{1cm}
\begin{center}
\begin{picture}(4,1.2)
\put(1,.5){\line(1,0){1}}
\put(1,.5){\circle*{.2}}
\put(2,.5){\circle*{.2}}
\put(2,.5){\line(1,0){1}}
\put(3,.5){\circle*{.2}}
\put(2,.5){\circle{.35}}
\put(.9,.8){$\gamma_1$}
\put(1.9,.8){$\alpha$}
\put(2.9,.8){$\beta_1$}
\end{picture}
\end{center}

Let $L(C)$ be a simple module in the category $\cal O _{\Phi,\theta}$. In this case, $\germ l '=A_1$. Therefore the module $C$ is of the form $C=N(a_1,a_2)$ with $a_{1}$ and $a_{2}$ two non integer complex numbers, and the center of $\germ l $ acts as scalar operators on $C$. Set $A=a_1+a_2$. Remark that $H_{\alpha}+2H_{\beta_1}$ and $H_{\alpha}+2H_{\gamma_1}$ are in the center of $\germ l _{\Phi\setminus\theta}$. Denote by $c(k)$ and $c'(k)$ the respective actions of $H_{\beta_1}$ and $H_{\gamma_1}$ on $x(k)$. As in lemma \ref{lemA12}, we get $c(k)=c+a_2+k$ and $c'(b)=c'+a_2+k$ where $c$ and $c'$ are equal either to $0$ or to $-1-A$.

\begin{lemm}\label{lem:A1N}
With the notations as above we have $c+c'+A+1=0$ and $cc'=0$.
\end{lemm}
\begin{proof} Consider $v:=p(X_{-\gamma_1-\alpha-\beta_1}\otimes x(k))$. Lemma \ref{lem7} applied to $\alpha$ and $\mbox{\boldmath $\beta$}=(\beta_1,\gamma_1)$ implies that there is a non zero complex number $\eta(k)$ such that
\begin{align}\label{lemA1n}
v= & \eta(k)p(X_{-\beta_1}X_{-\gamma_1}\otimes x(k-1)).
\end{align}
Apply $X_{\gamma_1+\alpha+\beta_1}$ to equation \eqref{lemA1n}. We get:
\begin{align*}
p([X_{\gamma_1+\alpha+\beta_1},X_{-\gamma_1-\alpha-\beta_1}] \otimes x(k)) = & \eta(k) p\left(\left([X_{\gamma_1+\alpha+\beta_1},X_{-\beta_1}]X_{-\gamma_1}\right.\right.\\
 & \left.\left. +X_{-\beta_1}[X_{\gamma_1+\alpha+\beta_1},X_{-\gamma_1}]\right) \otimes x(k-1)\right).
\end{align*}
On the other hand we have the following:
\begin{align*}
[X_{\gamma_1+\alpha+\beta_1},X_{-\gamma_1-\alpha-\beta_1}] = & H_{\beta_1+\alpha+\gamma_1}=H_{\beta_1}+H_{\alpha}+H_{\beta_1}\\
[X_{\gamma_1+\alpha+\beta_1},X_{-\beta_1}]= & X_{\gamma_1+\alpha}\\
[X_{\gamma_1+\alpha+\beta_1},X_{-\gamma_1}]= & -X_{\beta_1+\alpha}.
\end{align*}
Thus we obtain:
\begin{align*}
p((H_{\beta_1}+H_{\alpha}+H_{\beta_1})\otimes x(k)) = & \eta(k) p((X_{\alpha+\gamma_1}X_{-\gamma_1}-X_{-\beta_1}X_{\alpha+\beta_1})\otimes x(k-1)).
\end{align*}
But now $[X_{\alpha+\gamma_1},X_{-\gamma_1}]=-X_{\alpha}$ and $X_{\alpha+\beta_1}\in \germ n ^+$. Hence:
\begin{multline}\label{lemA1kn}
(c+c'+a_2+a_1)p(1\otimes x(k))=  -\eta(k)(a_2-k+1)p(1\otimes x(k)).
\end{multline}
Applying now $X_{\alpha+\beta_1}$ to equation \eqref{lemA1n}, we get:
$$p(X_{-\gamma_1}\otimes x(k))=\eta(k)(a_2-k+1)p(X_{-\gamma_1}\otimes x(b)).$$ Since $p(X_{-\gamma_1}\otimes x(b))\not= 0$ by lemma \ref{lem2}, we deduce that $\eta(k)(a_2-k+1)=1$. Then equation \eqref{lemA1kn} gives $c+c'+A+1=0$. Since $c$ and $c'$ are equal either to $0$ or to $-1-A$, we conclude that $cc'=0$ except if $c=c'=-1-A$. But in this case the equation $c+c'+A+1=0$ gives $-1-A=0$ and therefore $c=c'=0$.

\end{proof}

From now on, we will assume that $c=0$ and $c'=-1-A$. First we have the following corollary, whose proof is analogous to the proof of corollary \ref{corA12} (and is thus omitted).

\begin{coro}\label{cor:A1N}
Assume $\germ g =A_{3}$. Consider $\theta \subset \Phi$ such that $\Phi\setminus\theta=\{e_2\}$. Then $M$ is a simple module in the category $\cal O _{\Phi,\theta}$ if and only if $M$ is isomorphic to some $N(-1,a_1,a_2,0)$ with $a_1,a_2 \in \C\setminus\Z$.
\end{coro}

From now on we assume that $n>3$. Hence we have the following Dynkin diagram: 
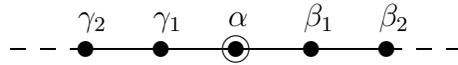
\setlength{\unitlength}{1cm}
\begin{center}
\begin{picture}(6,1.2)
\put(0,.5){\line(1,0){.2}}
\put(.4,.5){\line(1,0){.2}}
\put(.8,.5){\line(1,0){.2}}
\put(1,.5){\line(1,0){1}}
\put(2,.5){\line(1,0){1}}
\put(3,.5){\line(1,0){1}}
\put(4,.5){\line(1,0){1}}
\put(5,.5){\line(1,0){.2}}
\put(5.4,.5){\line(1,0){.2}}
\put(5.8,.5){\line(1,0){.2}}
\put(2,.5){\circle*{.2}}
\put(3,.5){\circle*{.2}}
\put(4,.5){\circle*{.2}}
\put(1,.5){\circle*{.2}}
\put(5,.5){\circle*{.2}}
\put(3,.5){\circle{.35}}
\put(.9,.8){$\gamma_2$}
\put(2.9,.8){$\alpha$}
\put(3.9,.8){$\beta_1$}
\put(1.9,.8){$\gamma_1$}
\put(4.9,.8){$\beta_2$}
\end{picture}
\end{center}
The vectors $H_{\beta_i}$ and $H_{\gamma_j}$ for $i$ and $j$ greater than $1$ are in the center of $\germ l $. Denote by $d_i$ and $d'_j$ their action on $C$. We show:

\begin{lemm}\label{lemA1NL}
With notations as above, we have $d_i=0=d'_j$ for any $i>1$ and $j>1$.
\end{lemm}
\begin{proof} We prove the lemma for the $d_i$'s. Assume the lemma does not hold. Let $i$ be the first integer such that $d_i\not=0$. Consider $v:=p(X_{-(\gamma_1+\alpha+\beta_1+\cdots +\beta_i)}\otimes x(k))$. Lemma \ref{lem7} applied to $\alpha$ and $\mbox{\boldmath $\beta$}=(\gamma_1,\beta_1+\cdots +\beta_i)$ gives $v\not=0$ and  
$$v \in p(\cal U (\germ l _{\theta}^-)_{-\gamma_1-\beta_1-\ldots -\beta_i}\otimes x(k-1)).$$ Since $X_{-\gamma_1}$ commutes with the $X_{-\beta_i}$'s, we have  
$$\cal U (\germ l _{\theta}^-)_{-\gamma_1-\beta_1-\ldots -\beta_i}=X_{-\gamma_1}\cal U (\germ l _{\theta}^-)_{-\beta_1-\ldots -\beta_i}.$$ Thus we get:
\begin{align}\label{lem2A1n}
v= & \sum_{\sigma \in \mathfrak{S}_i} \: \eta_{\sigma} p(X_{-\gamma_1}X_{-\sigma(\beta_1)}\cdots X_{-\sigma(\beta_i)}\otimes x(k-1)).
\end{align}
Apply the vectors $X_{\alpha+\beta_1+\cdots +\beta_i}$ and $X_{\gamma_1+\alpha+\beta_1+\cdots +\beta_i}$ to equation \eqref{lem2A1n}. Remark that $$X_{\alpha+\beta_1+\cdots +\beta_i}\cdot p(X_{-\gamma_1}X_{-\sigma(\beta_1)}\cdots X_{-\sigma(\beta_i)}\otimes x(k-1))=0$$ except if $\sigma=\sigma_0:=(i,i-1,\ldots ,1)$. Indeed $[X_{\alpha+\beta_1+\cdots +\beta_i},X_{-\beta_j}]=0$ except if $j=i$ and in this latter case we have $[X_{\alpha+\beta_1+\cdots +\beta_i},X_{-\beta_i}]=X_{\alpha+\beta_1+\cdots +\beta_{i-1}}$. Therefore we obtain:
\begin{multline}
X_{\alpha+\beta_1+\cdots +\beta_i}\cdot p(X_{-\gamma_1}X_{-\sigma_0(\beta_1)}\cdots X_{-\sigma_0(\beta_i)}\otimes x(k-1))= \\
p(X_{-\gamma_1}X_{\alpha}\otimes x(k-1)) = (a_2-k+1)p(X_{-\gamma_1}\otimes x(k)).
\end{multline}
Remark also that $$X_{\gamma_1+\alpha+\beta_1+\cdots +\beta_i}\cdot p(X_{-\gamma_1}X_{-\sigma(\beta_1)}\cdots X_{-\sigma(\beta_i)}\otimes x(k-1))=0 \mbox{ except if } \sigma=\sigma_0,$$ and in this latter case we have 
$$X_{\gamma_1+\alpha+\beta_1+\cdots +\beta_i}\cdot p(X_{-\gamma_1}X_{-\sigma_0(\beta_1)}\cdots X_{-\sigma_0(\beta_i)}\otimes x(k-1))=-p(1\otimes X_{\alpha}x(k-1)).$$ So we get:
\begin{multline}
X_{\gamma_1+\alpha+\beta_1+\cdots +\beta_i}\cdot p(X_{-\gamma_1}X_{-\sigma(\beta_1)}\cdots X_{-\sigma(\beta_i)}\otimes x(k-1)) \\ =  -(a_2-k+1)p(1\otimes x(k)).
\end{multline}
On the other hand we have the following:
$$[X_{\alpha+\beta_1+\cdots +\beta_i},X_{-(\gamma_1+\alpha+\beta_1+\cdots +\beta_i)}]=X_{-\gamma_1}$$ and 
$$[X_{\gamma_1+\alpha+\beta_1+\cdots +\beta_i},X_{-(\gamma_1+\alpha+\beta_1+\cdots +\beta_i)}]=H_{\gamma_1}+H_{\alpha}+H_{\beta_1}+\cdots +H_{\beta_i}.$$  Therefore we are left with the following equations:
$$\left\{ \begin{array}{l}
p(X_{-\gamma_1}\otimes x(k))=  \eta_{\sigma_0}(a_2-k+1)p(X_{-\gamma_1}\otimes x(k)),\\
(c'+a_1+c+a_2+d_i)p(1\otimes x(k))=  -\eta_{\sigma_0}(a_2-k+1)p(1\otimes x(k)).
\end{array}\right.$$ Since $p(X_{-\gamma_1}\otimes x(b))\not= 0$ according to lemma \ref{lem2}, we deduce from that $\eta_{\sigma_0}(b_2+1)=1$ and $d_i=0$. This contradicts our assumption.

\end{proof}

We have now the following corollary, whose proof is analogous to the proof of corollary \ref{corA12} (and is thus omitted).

\begin{coro}\label{corA1NL}
Let $\germ g =A_n$. Let $\theta\subset \Phi$ such that $\Phi\setminus\theta=\{e_l\}$ for some $1<l<n$. Then the simple module $M$ is in the category $\cal O _{\Phi,\theta}$ if and only if $M$ is isomorphic to some $N(-1,\ldots, -1,a'_1,a'_2,0,\ldots ,0)$ with $a'_{1},\: a'_{2}\in\C\setminus\Z$.
\end{coro}


\subsubsection{Case $\germ l '=\germ sl _{l+1}$, with $l>1$}\label{secAkAn}


We show the following:

\begin{theo}\label{thm:AkAn}
Let $\germ g =A_n$. Let $1<l<n$. Let $\theta\subset \Phi$ such that $\Phi\setminus\theta$ is a connected subset of $\Phi$ with cardinality $l$. A simple module $M$ is in the category $\cal O _{\Phi,\theta}$ if and only if it is isomorphic to some $N(-1,\ldots, -1,a_1,\ldots , a_{l+1},0,\ldots ,0)$ with $a_1,\ldots , a_{l+1} \in \C\setminus\Z$.
\end{theo}
\begin{proof} Let $L(C)$ be a simple module in $\cal O _{\Phi,\theta}$. We know then that $C$ is a simple cuspidal $\germ l $-module of degree $1$. Here $\germ l '=A_l$. Therefore $C$ is a the form $C=N(a_1,\ldots ,a_{l+1})$ with $a_{i}\in \C\setminus\Z$. Assume first that the Dynkin diagram of $\germ g $ contains the following piece:

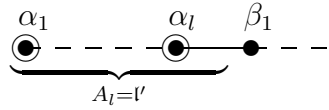
\setlength{\unitlength}{1cm}
\begin{center}
\begin{picture}(6,1.2)
\put(1,.5){\circle*{.2}}
\put(1,.5){\line(1,0){.2}}
\put(1.4,.5){\line(1,0){.2}}
\put(1.8,.5){\line(1,0){.2}}
\put(2.2,.5){\line(1,0){.2}}
\put(2.6,.5){\line(1,0){.2}}
\put(3,.5){\circle*{.2}}
\put(3,.5){\line(1,0){1}}
\put(4,.5){\circle*{.2}}
\put(4,.5){\line(1,0){.2}}
\put(4.4,.5){\line(1,0){.2}}
\put(4.8,.5){\line(1,0){.2}}
\put(1,.5){\circle{.35}}
\put(3,.5){\circle{.35}}
\put(3.9,.8){$\beta_1$}
\put(.9,.8){$\alpha_1$}
\put(2.9,.8){$\alpha_l$}
\put(.8,.3){$\underbrace{\hphantom{11111111111111}}_{A_l=\germ l '}$}
\end{picture}
\end{center}

Then the vector $H_{\alpha_1}+2H_{\alpha_2}+\cdots +lH_{\alpha_l}+(l+1)H_{\beta_1}$ is in the center of $\germ l $. Its action on $C$ must be constant. Denote by $c(k)$ the action of $H_{\beta_1}$ on $x(k)\in C$. We get $(a_{1}+k_1-a_2-k_{2})+2(a_2+k_{2}-a_3-k_{3})+\cdots +l(a_l+k_{l}-a_{l+1}-k_{l+1})+(l+1)c(k)=cte$. Since $k_1+\cdots +k_{l+1}=0$ we obtain $c(k)=c+a_{l+1}+k_{l+1}$ for some complex number $c$. Consider now $u:=p(X_{-\alpha_l-\beta_1}\otimes x(k))$. Note that $X_{-\alpha_l}x(k)=(a_l+k_{l})x(k-\epsilon_{l}+\epsilon_{l+1})$ and $X_{\alpha_l+\alpha_{l-1}}x(k)=(a_{l+1}+k_{l+1})x(k-\epsilon_{l+1}+\epsilon_{l-1})$. Apply now lemma \ref{lem7} to $u$. There is a non zero complex number $\eta(k)$ such that
\begin{eqnarray}\label{eq3}
u=\eta(k)p(X_{-\beta_1}\otimes x(k-\epsilon_{l}+\epsilon_{l+1})).
\end{eqnarray}
Apply the vectors $X_{\beta_1}$ and $X_{\alpha_{k-1}+\alpha_k+\beta_1}$ to the equation \eqref{eq3}. We get
\begin{multline}
(a_l+k_{l})p(1\otimes x(k-\epsilon_{l}+\epsilon_{l+1}))  =  \eta(k)c(k-\epsilon_{l}+\epsilon_{l+1})p(1\otimes x(k-\epsilon_{l}+\epsilon_{l+1})),\\
(a_l+k_{l})p(1\otimes x(k-\epsilon_{l+1}+\epsilon_{l-1}))  =  \eta(k)(a_{l+1}+k_{l+1}+1)p(1\otimes x(k-\epsilon_{l+1}+\epsilon_{l-1})).
\end{multline} Since $\eta(k)\not=0$, $p(1\otimes x(k-\epsilon_{l}+\epsilon_{l+1}))\not=0$ and $p(1\otimes x(k-\epsilon_{l+1}+\epsilon_{l-1}))\not=0$, we deduce from the above equations that $c=0$.

Assume now the Dynkin diagram of $\germ g $ contains the following piece:

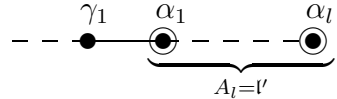
\setlength{\unitlength}{1cm}
\begin{center}
\begin{picture}(6,1.2)
\put(1,.5){\line(1,0){.2}}
\put(1.4,.5){\line(1,0){.2}}
\put(1.8,.5){\line(1,0){.2}}
\put(2,.5){\circle*{.2}}
\put(2,.5){\line(1,0){1}}
\put(3,.5){\circle*{.2}}
\put(3,.5){\line(1,0){.2}}
\put(3.4,.5){\line(1,0){.2}}
\put(3.8,.5){\line(1,0){.2}}
\put(4.2,.5){\line(1,0){.2}}
\put(4.6,.5){\line(1,0){.2}}
\put(5,.5){\circle*{.2}}
\put(3,.5){\circle{.35}}
\put(5,.5){\circle{.35}}
\put(1.9,.8){$\gamma_1$}
\put(2.9,.8){$\alpha_1$}
\put(4.9,.8){$\alpha_l$}
\put(2.8,.3){$\underbrace{\hphantom{111111111111}}_{A_l=\germ l '}$}
\end{picture}
\end{center}

Denote by $c'(k)$ the action of $H_{\gamma_1}$ on $x(k)\in C$. A reasoning as above shows that $c'(k)=c'-(a_1+k_{1})$ with $c'=-1$.

In general, we have the following Dynkin diagram:
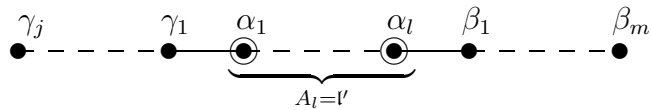
\setlength{\unitlength}{1cm}
\begin{center}
\begin{picture}(8,1.2)
\put(0,.5){\circle*{.2}}
\put(0,.5){\line(1,0){.2}}
\put(.4,.5){\line(1,0){.2}}
\put(.8,.5){\line(1,0){.2}}
\put(1.2,.5){\line(1,0){.2}}
\put(1.6,.5){\line(1,0){.2}}
\put(2,.5){\circle*{.2}}
\put(2,.5){\line(1,0){1}}
\put(3,.5){\circle*{.2}}
\put(3,.5){\line(1,0){.2}}
\put(3.4,.5){\line(1,0){.2}}
\put(3.8,.5){\line(1,0){.2}}
\put(4.2,.5){\line(1,0){.2}}
\put(4.6,.5){\line(1,0){.2}}
\put(5,.5){\line(1,0){1}}
\put(6,.5){\line(1,0){.2}}
\put(6.4,.5){\line(1,0){.2}}
\put(6.8,.5){\line(1,0){.2}}
\put(7.2,.5){\line(1,0){.2}}
\put(7.6,.5){\line(1,0){.2}}
\put(6,.5){\circle*{.2}}
\put(8,.5){\circle*{.2}}
\put(5,.5){\circle*{.2}}
\put(3,.5){\circle{.35}}
\put(5,.5){\circle{.35}}
\put(0,.8){$\gamma_j$}
\put(1.9,.8){$\gamma_1$}
\put(2.9,.8){$\alpha_1$}
\put(4.9,.8){$\alpha_l$}
\put(5.9,.8){$\beta_1$}
\put(7.9,.8){$\beta_m$}
\put(2.8,.3){$\underbrace{\hphantom{111111111111}}_{A_l=\germ l '}$}
\end{picture}
\end{center}
Note then that $H_{\beta_i}$ and $H_{\gamma_i}$ with $i >1$ belong to the center of $\germ l $. Thus they act as constant on $C$. Denote by $d_i$ and $d'_i$ these constants. We show that $d_{i}=0=d'_{i}$. Let us prove this for the $d_i$'s. Assume it does not hold. Let $p$ be the first integer such that $d_p\not=0$. We apply lemma \ref{lem7} to the vector $v:=p(X_{-\alpha_l-\beta_1-\ldots -\beta_p}\otimes x(k))$ with $\alpha=\alpha_l$ and $\mbox{\boldmath $\beta$}=(\beta_1+\cdots +\beta_p)$. We have $v\not=0$ and $v\in \cal U (\germ l _{\theta}^-)_{-(\beta_1+\cdots +\beta_p)}p(1\otimes x(k-\epsilon_{l}+\epsilon_{l+1}))$. More precisely we can write down the following expression:
$$v=\sum_{\sigma \in \mathfrak{S}_p}\: \eta_{\sigma}p(X_{-\beta_{\sigma(1)}}\cdots X_{-\beta_{\sigma(p)}}\otimes x(k-\epsilon_{l}+\epsilon_{l+1})).$$ Apply then the vectors $X_{\alpha_l+\beta_1+\cdots +\beta_p}$ et $X_{\alpha_{l-1}+\alpha_l+\beta_1+\cdots +\beta_p}$ to this equation. Note that both of these vectors act trivially on $p(X_{-\beta_{\sigma(1)}}\cdots X_{-\beta_{\sigma(p)}}\otimes x(k-\epsilon_{l}+\epsilon_{l+1}))$ except if $\sigma=\sigma_0=(p,p-1,\ldots ,1)$. Computations analogous to those in lemma \ref{lemA1NL} give the following two equations: 
\begin{multline}
((a_l+k_{l}-a_{l+1}-k_{l+1})+c(k)+d_p)p(1\otimes x(k))=  \eta_{\sigma_0}(a_{l+1}+k_{l+1}+1)p(1\otimes x(k))\\ 
(a_l+k_{l})p(1\otimes x(k-\epsilon_{l+1}+\epsilon_{l-1}))= \eta_{\sigma_0}(a_{l+1}+k_{l+1}+1)p(1\otimes x(k-\epsilon_{l+1}+\epsilon_{l-1})).
\end{multline}
We deduce then that $\eta_{\sigma_0}(a_{l+1}+k_{l+1}+1)=a_l+k_{l}$ and thus that $d_p=0$. This contradicts our assumption.

Now we conclude as in corollary \ref{corA1NL}. We compare the action of $\cal U (\germ g ) _0$ on the weight vector $x(0)$ with its action on the weight vector $x(0)\in N(a)$ with $a=(-1,\ldots ,-1,a_1,\ldots ,a_{l+1},0,\ldots ,0)$ and then use Lemire's correspondence \cite{Le68}.

\end{proof}


\subsection{Type $C$ case}\label{ssec:typeC}


In this section we assume that $\germ g =\germ sp _{2n}$.

\subsubsection{Case $\germ l $ of type $A$}\label{sssec:CA}

According to the theorem \ref{thm1red}, we need only to consider the following situation:

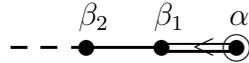
\setlength{\unitlength}{1cm}
\begin{center}
\begin{picture}(5,1.2)
\put(1,.5){\line(1,0){.2}}
\put(1.4,.5){\line(1,0){.2}}
\put(1.8,.5){\line(1,0){.2}}
\put(2,.5){\circle*{.2}}
\put(2,.5){\line(1,0){1}}
\put(3,.5){\circle*{.2}}
\put(3,0.45){\line(1,0){1}}
\put(3,.55){\line(1,0){1}}
\put(4,.5){\circle*{.2}}
\put(4,.5){\circle{.35}}
\put(1.9,.8){$\beta_2$}
\put(2.9,.8){$\beta_1$}
\put(3.9,.8){$\alpha$}
\put(3.4,.4){$<$}
\end{picture}
\end{center} 

Let $L(C)$ be a simple module in the corresponding category $\cal O _{\Phi,\theta}$. We know that $C$ is a simple cuspidal $\germ l $-module of degree $1$. Since $\germ l '=A_{1}$ we must have $C=N(a_1,a_2)$ for some $a_{1},\: a_{2}\in \C\setminus\Z$. Set $A=a_1+a_2$. The vector $H_{\beta_1}+H_{\alpha}$ is in the center of $\germ l $. Therefore it acts on $C$ by some constant. Denote by $c(k)$ the action of $H_{\beta_1}$ on $x(b)$. Then we find that $c(k)=c+2a_2-2k$.
\begin{lemm}\label{lemAC1}
With the notations as above, we have either $c=0$ or $c=-2-2A$. We also have $2c+2A+1=0$.
\end{lemm}
\begin{proof} The proof is analogous to that of lemma \ref{lemA12}. Set 
$$u:=p(X_{-\alpha-\beta_1}\otimes x(k)) \mbox{ and } v:=p(X_{-\alpha-2\beta_1}\otimes x(k)).$$ Lemma \ref{lem7} implies that $u=\eta(k)p(X_{-\beta_1}\otimes x(k-1))$ for some non zero complex number $\eta(k)$. Apply $X_{\beta_1}$ and $X_{\alpha+\beta_1}$ to the equation $u=\eta(k)p(X_{-\beta_1}\otimes x(k-1))$. Using the structure constants of $C_{2}$ we get:
$$\left\{ \begin{array}{rcl}
-2(a_1+k) p(1\otimes x(k-1))& = & \eta(k)c(k-1)p(1\otimes x(k-1))\\
(2(a_1-a_2+2k)+c(k))p(1\otimes x(k)) & = & -2\eta(k)(a_2-k+1)p(1\otimes x(k))
\end{array}\right.$$
Since $p(1\otimes x(k-1))\not=0$ and $p(1\otimes x(k))\not=0$, this system gives $c=0$ or $c=-2-2A$ and $\eta(k)=-\frac{c+2a_1+2k}{2a_2-2k+2}$.

Now lemma \ref{lem7} applied to $\alpha$ and $\mbox{\boldmath $\beta$}=(\beta_1,\beta_1)$ gives $v=\eta'(k)p(X_{-\beta_1}^2\otimes x(k-1))$ for a non zero complex number $\eta'(k)$. Apply $X_{\beta_1}$, $X_{\alpha+\beta_1}$ and $X_{\alpha+2\beta_1}$ to this equation to obtain the following system:
$$\left\{ \begin{array}{l}
-p(X_{-\alpha-\beta_1}\otimes x(k)) =  2\eta'(k)(c(k-1)-1) p(X_{-\beta_1}\otimes x(k-1))\\
p(X_{-\beta_1}\otimes x(k))  =  -4\eta'(k)(a_2-k+1) p(X_{-\beta_1}\otimes x(k))\\
(c(k)+a_1-a_2+2k)p(1\otimes x(k))  =  2\eta'(k)(a_2-k+1)p(1\otimes x(k))
\end{array} \right.$$
We solve this system using the value of $c$ and of $\eta(k)$ found above. We get $2c+2A+1=0$. 

\end{proof}

\begin{coro}\label{cor:AC1}
Assume $\germ g =C_{2}$ and $\theta=\{e_{1}\}$. A simple module $M$ belongs to the category $\cal O _{\Phi,\theta}$ if and only if it is isomorphic to $M(-1,a_1-a_2-\frac{1}{2})$.
\end{coro}
\begin{proof} Once more, it suffices to check that the action of $\cal U (\germ g )_0$ on $p(1\otimes x(0))\in L(N(a_{1},a_{2}))$ is identical as its action on $x(0)\in M(-1,a_1-a_2- \frac{1}{2})$. We then conclude with Lemire's correspondence \cite{Le68}.

\end{proof}

\begin{rema}
Note that if $c=0$ then lemma \ref{lemAC1} implies that $a_1-a_2-\frac{1}{2}=2a_1$ and if $c=-2-2A$, it implies that $a_1-a_2-\frac{1}{2}=2a_1+1$. In both cases we see that  $a_1-a_2-\frac{1}{2}$ is a non integer complex number. Conversely if $z\in \C\setminus \Z$, there is a pair $(a_{1},a_{2})$ of non integer complex numbers such that $z=a_1-a_2-\frac{1}{2}$ and $z=2a_1$ or $z=2a_1+1$.

We also note that $N(a_1,a_2)\cong N(-1-a_2,-1-a_1)$. But if $a_1+a_2=-\frac{1}{2}$ then $(-1-a_1)+(-1-a_2)=-\frac{3}{2}$ and if $a_1+a_2=-\frac{3}{2}$ then $(-1-a_1)+(-1-a_2)=-\frac{1}{2}$. Therefore with the notations as above, we can always assume that $c=0$ and $A=-\frac{1}{2}$.
\end{rema}

We now show the following:

\begin{theo}\label{thm:AC1}
Assume $\germ g =C_{n}$ with $n>2$ and $\Phi\setminus\theta=\{e_{n}\}$. A simple module $M$ is in the category $\cal O _{\Phi,\theta}$ if and only if it is isomorphic to $M(-1,\ldots ,-1,a)$ for some $a\in \C-\Z$.
\end{theo}
\begin{proof} Let $L(C)$ be a simple module in $\cal O _{\Phi,\theta}$. We keep the notations above. As we mentioned, we can assume that $c=0$ and $A=-\frac{1}{2}$. Now the vectors $H_{\beta_2},\ldots ,H_{\beta_{n-1}}$ are in the center of $\germ l $. Thus they act on $C$ by some constants. Denote $d_{2},\ldots ,d_{n-1}$ these constants. We show by induction that $d_{i}=0$. 

\begin{enumerate}
\item We begin with $d:=d_{2}$. Assume that $d\not=0$. Then we have $X_{\beta_2}\cdot (X_{-\beta_2}\otimes x(k))=H_{\beta_2}\otimes x(k)=d\times 1\otimes x(k)\not=0$. Hence, the  proposition \ref{GVM} implies that $p(X_{-\beta_2}\otimes x(b))\not=0$. Consider $u:=p(X_{-\alpha-\beta_1-\beta_2}\otimes x(k))$. We apply lemma \ref{lem7} to $\alpha$ et $\mbox{\boldmath $\beta$}=(\beta_1,\beta_2)$. We get $u\in p(\cal U (\germ l ^-_{\theta})_{-\beta_1-\beta_2}\otimes x(k-1))$. But $\cal U (\germ l ^-_{\theta})_{-\beta_1-\beta_2}$ is generated by the two vectors $X_{-\beta_1}X_{-\beta_2}$ and $X_{-\beta_2}X_{-\beta_1}$. Therefore there are two non zero complex numbers $\eta_{1}(k)$ and $\eta_{2}(k)$ such that
\begin{align}\label{eqAC1}
u= & \eta_1(k)p(X_{-\beta_2}X_{-\beta_1}\otimes x(k-1))+\eta_2(k)p(X_{-\beta_1}X_{-\beta_2}\otimes x(k-1)).
\end{align}
Apply to equation \eqref{eqAC1} the vectors $X_{\beta_1}$, $X_{\alpha+\beta_1}$ and $X_{\alpha+\beta_1+\beta_2}$. Using the structure constants of $C_n$, we have the following system:
$$\left\{ \begin{array}{l}
0  =  (\eta_1(k)c(k-1)+\eta_2(k)(c(k-1)+1))p(X_{-\beta_2}\otimes x(k-1))\\
p(X_{-\beta_2}\otimes x(k))  =  -2(\eta_1(k)+\eta_2(k))(a_2-k+1)p(X_{-\beta_2}\otimes x(k))\\
(2(a_1-a_2+2k)+c(k)+d)p(1\otimes x(k))  =  -\eta_1(k)(a_2-k+1)p(1\otimes x(k))
\end{array}\right.$$
Since $p(X_{-\beta_2}\otimes x(b'))\not=0$, $p(X_{-\beta_2}\otimes x(b))\not=0$ and $p(1\otimes x(b))\not=0$ the solution of this system gives in particular $d=1-3a_1-3k$. This contradicts the fact that $d$ is a constant. Therefore $d=0$.

\item By induction assume that $d_{i}=0$ for $2 \leq i\leq l-1$. We prove then that $d:=d_{l}$ is also zero. On the contrary, assume $d\not=0$.

We begin with a lemma.

\begin{lemm}
Let $1<i\leq j<l$. Then for all $v\in \cal U (\germ n ^+_{\Phi\setminus\theta})$ and all $u\in \cal U (\germ l _{\theta}^-)_{-(\beta_i+\cdots +\beta_j)}$, we have $v\cdot (u\otimes x(k))=0$.
\end{lemm}
\begin{proof} We prove it by induction on $j-i$. If $1<i=j<l$, we have $X_{\beta_i}\cdot X_{-\beta_i}\otimes x(k)=H_{\beta_i}\otimes x(k)=0$ since $d_{i}=0$. Now for $X\in \germ n ^+_{\Phi\setminus\theta}$ of weight $\beta\not=\beta_i$, either $\beta-\beta_i\not\in \cal R $ or $\beta-\beta_i$ is a positive root, belonging to $\langle \Phi\setminus\theta\rangle ^+$. In both cases the action of $X$ on $X_{-\beta_i}\otimes x(k)$ is trivial. In other words the action of $\cal U (\germ n ^+_{\Phi\setminus\theta})$ on $X_{-\beta_i}\otimes x(k)$ is trivial.

Assume the lemma holds if $j-i<m$. Let $1<i\leq j<l$ be such that $j-i=m$. Let $u\in \cal U (\germ l _{\theta}^-)_{-(\beta_i+\cdots +\beta_j)}$. Without loss of generality we can assume that $u$ is a monomial. Let $v\in \cal U (\germ n ^+_{\Phi\setminus\theta})$ be a weight vector of weight $\beta$. Then we have the following possibilities. 
\begin{enumerate}
\item If $ad(v)(u)=0$, then $v\cdot u\otimes x(k)=0$.
\item If $ad(v)(u)$ is a weight vector with a negative weight. Then we must have $ad(v)(u)\in \cal U (\germ l _{\theta}^-)$ since the weight of $v$ is positive. Note that this case can only happen if $v\in \cal U (\germ l _{\theta})^+$. Then the weight of $ad(v)(u)$ is of smaller length that the weight of $u$. Thus we can apply the induction hypothesis to conclude that $v\cdot u\otimes x(k)=0$.
\item If $ad(v)(u)$ has zero weight, then $ad(v)(u)$ must be proportional to $H_{\beta_{i}}+\cdots +H_{\beta_{j}}$. Then the hypothesis $d_{p}=0$ for $1<p<l$ imply that $v\cdot u\otimes x(k)=0$.
\item If $ad(v)(u)$ is a weight vector of positive weight, then $ad(v)(u)\in \cal U (\germ n ^+_{\Phi\setminus\theta})$. Therefore we have $v\cdot u\otimes x(k)=0$.
\end{enumerate}
Thus in every case the action of $v$ on $u\otimes x(k)$ is trivial.

\end{proof}

The lemma together with proposition \ref{GVM} now imply that $p(u\otimes x(k))=0$ for all $u\in \cal U (\germ l _{\theta}^-)_{-(\beta_i+\cdots +\beta_j)}$. Consider now $$v=p(X_{-(\alpha+2\beta_1+\beta_2+\cdots +\beta_l)}\otimes x(k)).$$ Lemma \ref{lem7} with $\alpha$ and $\mbox{\boldmath $\beta$}=(\beta_1+\cdots +\beta_l,\beta_1)$ implies that $v\not=0$ and that $v\in p(\cal U (\germ l ^-_{\theta})_{-(2\beta_1+\beta_2+\cdots +\beta_l)}\otimes x(k-1))$. We shall order the elements in $\cal U (\germ l ^-_{\theta})_{-(2\beta_1+\beta_2+\cdots +\beta_l)}$ putting on the left the weight vectors in $\germ l _{\theta}^-$ whose weight has the form $\beta_l+\cdots +\beta_i$ and then we order the remaining vectors according to the length of their weight (that is the number of simple roots involved in the writing of the weight). Using the above lemma giving the non zero contributions in $p(\cal U (\germ l ^-_{\theta})_{-(2\beta_1+\beta_2+\cdots +\beta_l)}\otimes x(k-1))$, we finally get complex numbers $\mu_i$ such that
\begin{multline}\label{eqcx}
v= \mu_1p(X_{-(\beta_1+\cdots +\beta_l)}X_{-\beta_1}\otimes x(k-1))+\mu_2p(X_{-(\beta_2+\cdots +\beta_l)}X_{-\beta_1}^2\otimes x(k-1))+\\ \mu_3p(X_{-(\beta_3+\cdots +\beta_l)}X_{-\beta_1-\beta_2}X_{-\beta_1}\otimes x(k-1))\\
+\cdots +\mu_lp(X_{-\beta_l}X_{-(\beta_1+\cdots +\beta_{n-1})}X_{-\beta_1}\otimes x(k-1)).
\end{multline}
Apply $X_{\alpha+\beta_1+\cdots +\beta_k}$ to equation \eqref{eqcx}. We get
$$p(X_{-\beta_1}\otimes x(k))=2(-\mu_1+2\mu_2+\mu_3+\cdots +\mu_l)(a_2-k+1)p(X_{-\beta_1}\otimes x(k)).$$
Now we remark that the above lemma implies 
$$p(X_{-(\beta_2+\cdots +\beta_i)}X_{-\beta_1}\otimes x(k-1))=p(X_{-(\beta_1+\cdots +\beta_i)}\otimes x(k-1)), \quad i<l.$$ Indeed, we have
\begin{align*}
p(X_{-(\beta_2+\cdots +\beta_i)}X_{-\beta_1}\otimes x(k-1)) & =  p(X_{-\beta_1}X_{-(\beta_2+\cdots +\beta_i)}\otimes x(k-1))\\
 & +p([X_{-(\beta_2+\cdots +\beta_i)},X_{-\beta_1}]\otimes x(k-1))\\
 & = X_{-\beta_1}p(X_{-(\beta_2+\cdots +\beta_i)}\otimes x(k-1))\\
 & + p(X_{-(\beta_1+\cdots +\beta_i)}\otimes x(k-1))\\
 & = 0+p(X_{-(\beta_1+\cdots +\beta_i)}\otimes x(k-1))
\end{align*}
Apply now the vectors $X_{\beta_1}^2, X_{\beta_1}X_{\beta_2}, \ldots , X_{\beta_1}X_{\beta_{l-1}}$ to equation \eqref{eqcx}. Using the above lemma and the structure constants in $C_{n}$ we get:
$$
\left\{ \begin{array}{l}
0  =  \left(-2c(k-1)\mu_1+2c(k-1)(c(k-1)-1)\mu_2\right)p(X_{-(\beta_2+\cdots +\beta_l)}\otimes x(k-1))\\
0  =  \left(-2(c(k-1)-1)\mu_2+2(c(k-1)-1)\mu_3\right)p(X_{-(\beta_3+\cdots +\beta_l)}X_{-\beta_1}\otimes x(k-1))\\
0  =  \left(-(c(k-1)-1)\mu_3+(c(k-1)-1)\mu_4\right)p(X_{-(\beta_4+\cdots +\beta_l)}X_{-\beta_2-\beta_1}\otimes x(k-1))\\
\phantom{00} \vdots  \\
0  =  \left(-(c(k-1)-1)\mu_{l-1}+(c(k-1)-1)\mu_l\right)p(X_{-\beta_l}X_{-(\beta_1+\cdots \beta_{l-2}}\otimes x(k-1))
\end{array}\right.$$
On the other hand since $d\not=0$, we have for $i>1$: 
\begin{multline*}
X_{\beta_1+\cdots +\beta_{i-2}}X_{\beta_i+\cdots +\beta_l}\cdot X_{-(\beta_i+\cdots +\beta_l)}X_{-(\beta_1+\cdots +\beta_{i-2})}\otimes x(k-1)=\\
d\times c(k-1)\times 1\otimes x(k-1)\not= 0.
\end{multline*}
Proposition \ref{GVM} thus implies that
$$p(X_{-(\beta_i+\cdots +\beta_l)}X_{-(\beta_1+\cdots +\beta_{i-2})}\otimes x(k-1))\not=0.$$ Finally we apply the vector $X_{\beta_1}X_{\beta_1+\cdots +\beta_l}$ to equation \eqref{eqcx} to obtain:
\begin{multline}
2(a_1+k)p(1\otimes x(k-1)) = \left( ((d+c(k-1)-1)\mu_1+2(c(k-1)-1)\mu_2\right. \\
 \left. +(c(k-1)-1)\mu_3+\cdots +(c(k-1)-1)\mu_l)c(k-1)\right) p(1\otimes x(k-1)).
\end{multline}
We solve this system in the indeterminate $(\mu_1,\ldots ,\mu_l,d)$ obtained from the $k+1$ equations above, using $c(k)=c+2(a_2-k)$, $c=0$ and $A=-\frac{1}{2}$. In particular we get $d=1-2a_2+2k-l\times \frac{2a_2-2k+3}{2a_2-2k+1}$. This contradicts the fact that $d$ is a constant. Therefore we proved that $d=0$.
\item We conclude as in corollary \ref{corA1NL} that $L(C)$ is isomorphic to some $M(-1,\ldots ,-1,a)$ with $a\in \C\setminus \Z$.
\item Conversely, theorem \ref{thmMO} ensures that the module $M(-1,\ldots ,-1,a)$ with $a\in \C\setminus \Z$ is in the category $\cal O _{\Phi,\theta}$.
\end{enumerate}

\end{proof}

\subsubsection{Case $\germ l $ of type $C$}\label{sssec:CC}

We show the following:

\begin{theo}\label{thmCC}
Let $\theta\subset \Phi$ be such that $\Phi\setminus\theta$ is of type $C$. The simple module $M$ is in the category $\cal O _{\Phi,\theta}$ if and only if $M$ is isomorphic to $M(-1,\ldots,-1,a_{1},\ldots ,a_{l})$ with $a_{i}\in \C\setminus \Z$ and $l>1$.
\end{theo}
\begin{proof} \begin{enumerate}
\item Let us begin with the following case:
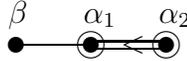
\setlength{\unitlength}{1cm}
\begin{center}
\begin{picture}(4,1.2)
\put(1,.5){\circle*{.2}}
\put(1,.5){\line(1,0){1}}
\put(2,.5){\circle*{.2}}
\put(2,0.45){\line(1,0){1}}
\put(2,.55){\line(1,0){1}}
\put(3,.5){\circle*{.2}}
\put(2,.5){\circle{.35}}
\put(3,.5){\circle{.35}}
\put(.9,.8){$\beta$}
\put(1.9,.8){$\alpha_1$}
\put(2.9,.8){$\alpha_2$}
\put(2.4,.4){$<$}
\end{picture}
\end{center} 

Let $L(C)$ be a simple module in $\cal O _{\Phi,\theta}$. We already know that $C$ is a simple cuspidal $\germ l $-module of degree $1$. Thus it is of the form $M(a)$ where $a\in (\C\setminus\Z)^2$. Denote by $c(k)$ the action of $H_{\beta}$ on $x(k)$. Since $H_{\beta}+H_{\alpha_1}+H_{\alpha_2}$ is in the center of $\germ l $, we deduce that $c(k)=c-a_1-k_{1}$ for some complex number $c$. Let us prove the following:
\begin{lemm}
We have $c=-1$.
\end{lemm}
\begin{proof} Consider $u:=p(X_{-\alpha_1-\beta}\otimes x(k))$. Lemma \ref{lem7} implies that there is a non zero complex number $\eta(k)$ such that 
$$u=\eta(k) p(X_{-\beta}\otimes x(k-\epsilon_{1}+\epsilon_{2})).$$ Apply now the vector $X_{\beta+\alpha_2+\alpha_1}$ to this equation. We get:
\begin{multline*}
p([X_{\beta+\alpha_1+\alpha_2},X_{-\beta-\alpha_1}]\otimes x(k))=\\
\eta(k)p([X_{\beta+\alpha_2+\alpha_1},X_{-\beta}]\otimes x(k-\epsilon_{1}+\epsilon_{2}).
\end{multline*}
Using the structure constants in $C_{3}$ we finally obtain $1=\eta(k)$. Apply then the vectors $X_{\beta}$ and $X_{\beta+\alpha_1}$ to the equation $u=\eta(k) p(X_{-\beta}\otimes x(k-\epsilon_{1}+\epsilon_{2}))$. We get two equations whose resolution gives $c=-1$.

\end{proof}

We can now check as in corollary \ref{corA1NL} that the $C_3$-module $L(C)$ is isomorphic to some $M(-1,a,b)$ for well chosen non integer complex numbers $a$ and $b$.

\item More generally if $\germ l '=C_2$, we show that the action of $H_{\beta_i}$ is trivial for $i>1$. This can be done by induction as in theorem \ref{thm:AkAn}. Here we need to consider the vector $v=p(X_{-(\alpha_1+\beta_1+\cdots +\beta_i)}\otimes x(k))$ and apply to $v$ the vectors $X_{\alpha_2+\alpha_1+\beta_1+\cdots +\beta_i}$, $X_{\alpha_1+\beta_1+\cdots +\beta_i}$, \ldots We left the details to the reader.

\item Assume now that $\germ l '=C_l$ for $l>2$. Then the Dynkin diagram contains the Dynkin diagram of $(A_{l-1},A_{n-1})$ of section \ref{secAkAn}. Therefore theorem \ref{thm:AkAn} gives us the action of the center of $\germ l $ on the module $C$. It only remains to check that $L(C)$ is then isomorphic to some $M(-1,\ldots ,-1,a_1,\ldots a_l)$ using the argument of corollary \ref{corA1NL}.

\item Conversely, the theorem \ref{thmNO} ensures that the simple modules \newline $M(-1,\ldots,-1,a_{1},\ldots ,a_{l})$ belong to some category $\cal O _{\Phi,\theta}$.
\end{enumerate}

\end{proof}


\subsection{Classification theorem}\label{ssec:class}


Let us gather the results obtained in this part:

\begin{theo}\label{thmpcpl}
Let $\germ g $ be a simple Lie algebra. Let $\theta \subset \Phi$ be such that $\theta\not=\Phi$ and $\theta\not=\emptyset$. Assume that $(\germ g , \Phi\setminus\theta)$ does not belong to table \ref{tab2}. Then:
\begin{enumerate}
\item The category $\cal O _{\Phi,\theta}$ is non trivial if and only if either $\germ g $ is isomorphic to $A_n$ and $\germ l '_{\Phi\setminus\theta}$ to $A_m$ with $m<n$ or $\germ g $ is isomorphic to $C_n$ and $\germ l '_{\Phi\setminus\theta}$ is isomorphic to either the $\germ {sl} _2$-subalgebra generated by the long simple root or to $C_k$ with $k<n$.
\item For a pair $(\germ g ,\germ l '_{\Phi\setminus\theta})$ as in $(1)$, the simple module in the category $\cal O _{\Phi,\theta}$ are of degree $1$ except in the case where $\germ g =A_n$ with $n>2$ and $\germ l '_{\Phi\setminus\theta}$ is the $\germ sl _{2}$-subalgebra generated by one of the simple roots on the extremity of the Dynkin diagram of $\germ g $. 
\item Conversely every simple module of degree $1$ belongs to one category $\cal O _{\Phi,\theta}$.
\end{enumerate}
\end{theo}

\begin{table}[htbp]
\begin{center}
\begin{tabular}{|c|c|}
\hline
\textbf{Type} & $\mathbf{\Phi\setminus\theta}$\\
\hline
$B_n$ & $\{e_1\}$\\
\hline
$D_n$ & $\{e_1\}$ or $\{e_{n-1}\}$ or $\{e_n\}$\\
\hline
$E_6$ & $\{e_1\}$ or $\{e_6\}$\\
\hline
$E_7$ & $\{e_7\}$\\
\hline
\end{tabular}
\caption{Excluded Cases}\label{tab2}
\end{center}
\end{table}

\begin{rema}
If $\germ g =A_n$ with $n>2$ and $\germ l '_{\Phi\setminus\theta}$ is the $\germ sl _{2}$-subalgebra generated by one of the simple roots on the extremity of the Dynkin diagram of $\germ g $, then there exists in $\cal O _{\Phi,\theta}$ simple modules not of degree $1$ (see appendix \ref{append}).
\end{rema}


\section{Semisimplicity of the category $\mathcal O _{\Phi,\theta}$}\label{sec:semis}


In this part we show that the non empty category $\cal O _{\Phi,\theta}$ is semisimple except if $\germ l '_{\Phi\setminus\theta}=\{e_{1}\} \mbox{ or } \{e_{n}\}$ when $\germ g $ is of type $A$. As in the previous part, we denote $\germ l :=\germ l _{\Phi\setminus\theta}$.


\subsection{Type $A$ case}\label{ssec:semiA}


In this section we assume that $\germ g =A_{n}$ for $n>1$. Let $\theta \subset \Phi$ be such that $\germ l $ is of type $A$. Moreover, if $n>2$ we will assume that $\Phi\setminus\theta\not=\{e_1\}$ and $\Phi\setminus\theta\not=\{e_n\}$.

\subsubsection{Case $\germ l '=A_{1}$}\label{sssec:semiA1}

Suppose here that $\Phi\setminus\theta=\{\alpha\}=\{\alpha_{l+1}\}$. According to the classification theorem \ref{thmpcpl}, the simple modules in $\cal O _{\Phi,\theta}$ are the modules $$N_a=N(\underbrace{-1,\ldots ,-1}_{l},a_1,a_2,\underbrace{0\ldots ,0}_{m}),\: l+2+m=n$$ where $a_i\in \C\setminus\Z$.

\begin{theo}\label{thmextA}
For $a$ and $b$ in $(\C\setminus\Z)^2$, we have $Ext^1(N_b,N_a)=\{0\}$.
\end{theo}

\begin{rema}
We shall prove the following equivalent statement:

Any exacte sequence $$(S_{a,b}) : 0\rightarrow N_a\rightarrow M\rightarrow N_b\rightarrow 0$$ with $M\in\cal O _{\Phi,\theta}$ splits.
\end{rema}

\begin{proof} Let us denote by $c$ the cocyle corresponding to the exact sequence $(S_{a,b})$. Note first that the sequence splits as a sequence of $\germ l _{\theta}$-modules since the restriction condition of category $\cal O _{\Phi,\theta}$ expresses that $N_{a}$, $N_{b}$ and $M$ are semisimple $\germ l _{\theta}$-modules. In other words, the cocyle $c$ vanishes on $\germ l _{\theta}$. To prove the theorem we have to show it vanishes on the whole Lie algebra $\germ g $. Thanks to the cocyle relation, we only need to show that $c(X_{\pm\alpha})=0$.

\begin{enumerate}
\item {\bf Suppose that $N_a \not\cong N_b$.} From theorem \ref{thmNO}, we know that the $\germ l _{\theta}$-highest weight vectors in $N_a$ or $N_b$ are the $x(k)$ with  $k_i=0$ for $i\not\in\{l+1,l+2\}$. To avoid any confusion we shall write $x(k)$ the basis vectors for $N_b$ and $y(j)$ for $N_a$. Using the explicit action of $\germ h $, we show:

\begin{lemm}
Assume $x(k)$ and $y(j)$ have the same weight under the action of $\germ h $. Then there is an integer $K$ such that
\begin{multline*}
y(j)=y( k_1+K,\ldots ,k_l+K,b_1-a_{1}+k_{l+1}+K,\\
 b_{2}-a_2+k_{l+2}+K,k_{l+3}+K,\ldots ,k_n+K)
\end{multline*}
and 
$$K+k_1\leq 0,\ldots ,K+k_l\leq 0,k_{l+3}+K\geq 0, \ldots , k_n+K\geq 0.$$ In particular, if $x(k)$ is a $\germ l _{\theta}$-highest weight vector and if $y(j)$ has the same weight as $x(k)$, then we have $K=0$ and $N_a \cong N_b$.
\end{lemm}

Recall from proposition \ref{propNM} that as $\germ l $-modules $N_a$ and $N_b$ are semisimple. From proposition \ref{prop:GS}, we get $c(X_{-\alpha})=0$ on every simple $\germ l '$--submodules of $N_b$ and therefore on the whole of $N_b$. Let $x=x(k)\in N_b$ be a $\germ l _{\theta}$-highest weight vector, of weight $\lambda$. Then we see that $c(X_{\alpha})(x)$, if non zero, is a weight vector of $N_a$ of weight $\lambda+\alpha$. Indeed, since $c(\germ h )=0$, we have
\begin{align*}
\underbrace{c([H_{\alpha},X_{\alpha}])(x)}_{2c(X_{\alpha})(x)} & =\underbrace{[c(H_{\alpha}),X_{\alpha}](x)}_{0}+[H_{\alpha},c(X_{\alpha})](x)\\
 & = H_{\alpha}\cdot c(X_{\alpha})(x)-\lambda(H_{\alpha})c(X_{\alpha})(x).
\end{align*}
Since the module $N_a$ is $\{-\alpha\}$--cuspidal, we deduce from the above equation that $N_a$ should have a weight vector of weight $\lambda$, that is with the same weight as $x(k)$. This contradicts the above lemma. Hence we must have $c(X_{\alpha})(x)=0$.

Consider now $\beta \in \langle \theta\rangle ^+$. Let $y=X_{-\beta}\cdot x$. Since $[X_{\alpha},X_{-\beta}]=0$, the cocyle relation together with the fact that $c(X_{-\beta})=0$ (for $X_{-\beta}\in \germ l _{\theta}$), show that 
$$0=c([X_{\alpha},X_{-\beta}])(x)=c(X_{\alpha})(y)-X_{-\beta}\cdot c(X_{\alpha})(x).$$ Since we proved that $c(X_{\alpha})(x)=0$, we must have $c(X_{\alpha})(y)=0$ as well. The very same computation implies that if for some vector $x'\in N_b$ we have $c(X_{\alpha})(x')=0$ then we must also have $c(X_{\alpha})(X_{-\beta}\cdot x')=0$. Since $N_b$ is a direct sum of simple $\germ l _{\theta}$--modules, reasoning by induction shows that $c(X_{\alpha})=0$ on $N_b$, which proves the theorem in this case.

\item {\bf Suppose that $N_a\cong N_b$.} Recall the action of $X^-=X_{-\alpha}$ and $X^+=X_{\alpha}$ on $N_a$ :
$$\left\{\begin{array}{ccc}
X^+\cdot x(k) & = & (a_2+k_{l+2})x(k-\epsilon_{l+2}+\epsilon_{l+1})\\
X^-\cdot x(k) & = & (a_1+k_{l+1})x(k-\epsilon_{l+1}+\epsilon_{l+2})
\end{array}\right.$$
From that we deduce that
$$\left(X^-\right)^{-1}\cdot x(k)=\frac{1}{a_1+k_{l+1}+1}x(k-\epsilon_{l+2}+\epsilon_{l+1}).$$
Proposition \ref{propNM} implies that $N_a$ is a direct sum of simple cuspidal $\germ l $--modules. Using the above equations, remark that the $\germ l $--module $\cal U (\germ l )x(k)$ does not depend upon $k_{l+1}$ and $k_{l+2}$. Proposition \ref{prop:GS} gives the expression of the cocycle $c$ on each module $\cal U (\germ l )x(k)$ : $c(X^-)=0$ and $c(X^+)=b(k)\times {\left(X^-\right)}^{-1}$ with $b(k)\in \C$. The previous remark implies that $b(k)$ should also not depend upon $k_{l+1}$ and $k_{l+2}$.

Let us come back to the notations of section \ref{subsec:A1N}. Let $\beta= \beta_1+\cdots +\beta_i\in \langle \theta\rangle^+$. Then we have $\alpha+\beta\in\cal R $, $[X^+,X_{\beta}]=X_{\alpha+\beta}$ and $[X_{\alpha+\beta},X_{-\beta}]=X^+$. Using the cocyle relation we thus get $c(X^+)=[[c(X^+),X_{\beta}],X_{-\beta}]$ (since $c(Y)=0$ for $Y\in \germ l _{\theta}$). This is an equality in $End_{\C}(N_a)$. Let us see how it acts on a vector $x(k)\in N_a$. Recall we have
$$\left\{\begin{array}{rcl}
X_{\beta}x(k) & = & k_{l+2+i}x(k+\epsilon_{l+2}-\epsilon_{l+2+i})\\
X_{-\beta}x(k) & = & (a_2+k_{l+2})x(k+\epsilon_{l+2+i}-\epsilon_{l+2})
\end{array}\right.$$

Thus we get
\begin{equation}\label{eqqqqq}
c(X^+)x(k)=\frac{b(k)}{a_1+k_{l+1}+1}x(k-\epsilon_{l+2}+\epsilon_{l+1}).
\end{equation}
We deduce then
$$[c(X^+),X_{\beta}]x(k)=\frac{k_{l+2+i}}{a_1+k_{l+1}+1}\left( b(k-\epsilon_{l+2+i})-b(k)\right) x(k-\epsilon_{l+2+i}+\epsilon_{l+1})$$ and  
\begin{multline}\label{eqcxA}
[[c(X^+),X_{\beta}],X_{-\beta}]x(k)=\\ \frac{a_2+k_{l+2}}{a_1+k_{l+1}+1}\Big{(}(k_{l+2+i}+1)(b(k)-b(k+\epsilon_{l+2+i})) \\
 -k_{l+2+i}(b(k-\epsilon_{l+2+i})-b(k))\Big{)}x(k-\epsilon_{l+2}+\epsilon_{l+1}).
\end{multline}
The equality of equation \eqref{eqqqqq} and equation \eqref{eqcxA} implies 
\begin{multline}\label{eq1666}
b(k)=  (k_{l+2+i}+1)(a_2+k_{l+2})(b(k)-b(k+\epsilon_{l+2+i}))\\ -k_{l+2+i}(a_2+k_{l+2})(b(k-\epsilon_{l+2+i})-b(k)).
\end{multline}
Let $k$ be such that $k_{l+2+i}=0$. Then we have $[c(X^+),X_{\beta}]x(k)=0$ and therefore we must have $c(X^+)x(k)=0$ which implies that $b(k)=0$. Now if $k$ is such that $k_{l+2+i}=1$, equation \eqref{eq1666} gives $2b(k+\epsilon_{l+2+i})=b(k)\left(3+\frac{1}{a_{2}+k_{l+2}}\right)$. This is a contradiction since $b(k)$ does not depend upon $k_{l+2}$, unless $b(k)=0=b(k+\epsilon_{l+2+i})$. Now a simple induction using equation \eqref{eq1666} shows that $b(k+j\epsilon_{l+2+i})=0$ for any non negative integer $j$. The same reasoning with the roots $\gamma=\gamma_1+\cdots +\gamma_i$ finally implies that $b(k)=0$. Hence the cocyle $c$ is zero, as asserted.
\end{enumerate}

\end{proof}

\begin{rema}
Looking at the proof above, we see that the theorem holds also in case $l=0$, that is $\Phi\setminus\theta=\{\alpha_{1}\}$. We will use this case of the theorem to prove our next result.
\end{rema}

\subsubsection{Case $\germ l '=A_{m-1}$, with $m>2$}\label{sssec:semiAn}

Let $\theta\subset \Phi$ be such that $card(\Phi\setminus\theta)>1$. Recall from theorem \ref{thmpcpl} that $\Phi\setminus\theta$ is a connected part of the Dynkin diagram, which thus looks as follows:
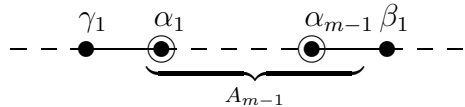
\setlength{\unitlength}{1cm}
\begin{center}
\begin{picture}(7,1.2)
\put(1,.5){\line(1,0){.2}}
\put(1.4,.5){\line(1,0){.2}}
\put(1.8,.5){\line(1,0){.2}}
\put(2,.5){\circle*{.2}}
\put(2,.5){\line(1,0){1}}
\put(3,.5){\circle*{.2}}
\put(3,.5){\line(1,0){.2}}
\put(3.4,.5){\line(1,0){.2}}
\put(3.8,.5){\line(1,0){.2}}
\put(4.2,.5){\line(1,0){.2}}
\put(4.6,.5){\line(1,0){.2}}
\put(5,.5){\line(1,0){1}}
\put(6,.5){\line(1,0){.2}}
\put(6.4,.5){\line(1,0){.2}}
\put(6.8,.5){\line(1,0){.2}}
\put(5,.5){\circle*{.2}}
\put(6,.5){\circle*{.2}}
\put(3,.5){\circle{.35}}
\put(5,.5){\circle{.35}}
\put(1.9,.8){$\gamma_1$}
\put(2.9,.8){$\alpha_1$}
\put(4.9,.8){$\alpha_{m-1}$}
\put(5.9,.8){$\beta_1$}
\put(2.8,.3){$\underbrace{\hphantom{11111111111111}}_{A_{m-1}}$}
\end{picture}
\end{center}

The simple modules in $\cal O _{\Phi,\theta}$ are the  
$$N_a=N(\underbrace{-1,\ldots ,-1}_j,a_1,\ldots ,a_m,\underbrace{0,\ldots ,0}_l),$$ with $a=(a_{1},\ldots ,a_{m})\in (\C\setminus\Z)^m$. Without lack of generality we can assume that $l$ is positive. We prove

\begin{theo}\label{thmextAb}
With the notations as above, we have $Ext^1(N_b,N_a)=\{0\}$.
\end{theo}
\begin{proof} Denote by $c$ the cocycle associated with the exact sequence 
$$0\rightarrow N_a \rightarrow M\rightarrow N_b\rightarrow 0.$$ We have to prove that $c=0$. From the restriction condition of $\cal O _{\Phi,\theta}$ we already know that $c$ is zero on $\germ l _{\theta}$. The cocyle relation implies then that it suffices to prove $c(X_{\pm \alpha_{i}})=0$ for $i\in \{1,\ldots ,m\}$. Our strategy is to use the previous theorem.

Let $\bar{\theta}=\{\beta_1,\ldots, \beta_l\}$. Consider the set $\theta_{1}=\bar{\theta}\cup\{\alpha_{m-1}\}$. Denote by $\germ l _{1}$ the corresponding Levi subalgebra, whose semisimple part is of type $A_{l+1}$. An explicit computation analogous to that in proposition \ref{propNM} shows that $N_a$ splits into a direct sum of simple $\germ l ^1$-modules. More precisely, these modules are isomorphic to $N(c_1,c_2,0,\ldots ,0)$ for some $c_1,c_2\in \C\setminus\Z$. In other words $N_a$ and $N_b$ are direct sum of simple modules in the category $\cal O _{\theta_{1},\bar{\theta}}$. Obviously $M$ belongs to this category by restriction. Theorem \ref{thmextA} implies now that $c(X_{\alpha_{m-1}})=0$ and $c(X_{-\alpha_{m-1}})=0$. 

We can reason the same way using $\theta_{2}=\bar{\theta}\cup\{\alpha_{m-2}+\alpha_{m-1}\}$. Hence we get $c(X_{\alpha_{m-2}+\alpha_{m-1}})=0$ and $c(X_{-(\alpha_{m-2}+\alpha_{m-1})})=0$. Applying the cocycle relation to $[X_{\alpha_{m-1}+\alpha_{m-2}},X_{-\alpha_{m-1}}]=X_{\alpha_{m-2}}$, we obtain $c(X_{\alpha_{m-2}})=0$ and $c(X_{-\alpha_{m-2}})=0$. Using the sets $\theta_{p}=\bar{\theta}\cup\{\alpha_{m-p}+\cdots +\alpha_{m-1}\}$ we deduce that the cocyle $c$ is zero, as asserted.

\end{proof}

Theorem \ref{thmextA} together with theorem \ref{thmextAb} leads to the following

\begin{coro}\label{corextAf}
Let $\germ g  =A_n$. Let $\theta$ be a proper subset of $\Phi$ such that $\Phi\setminus\theta$ is different from $\{e_1\}$ and $\{e_n\}$. Then the category $\cal O _{\Phi,\theta}$ is semisimple.
\end{coro}
\begin{proof}
Proposition \ref{prop:JH} implies that every module in the category $\cal O _{\Phi,\theta}$ admits a Jordan-H\"older series of finite length. We show the corollary by induction on the length $\ell (M)$ of $M$. If $\ell (M)=1$ then $M$ is a simple module. If $\ell (M)=2$ then there is an exact sequence 
$$0\rightarrow N_{a}\rightarrow M\rightarrow N_{b}\rightarrow 0.$$ By theorem \ref{thmextA} or theorem \ref{thmextAb} this sequence splits and thus $M=N_{a}\oplus N_{b}$. Assume now that the corollary holds for the modules in category $\cal O _{\Phi,\theta}$ of length at most $m$. Let $M$ be a module in $\cal O _{\Phi,\theta}$ of length $m+1$. Then there is a submodule $M'$ of $M$ of length $m$. This submodule is semisimple by induction. Now we have an exact sequence
$$0\rightarrow M'\rightarrow M\rightarrow N_{a}\rightarrow 0.$$ This must splits, which proves that $M=M'\oplus N_{a}$. This completes the proof.

\end{proof}

\begin{rema}
If $\theta=\Phi$, the category $\cal O _{\Phi,\Phi}(\germ sl _n)$ is semisimple by definition. If $\theta=\emptyset$, the category $\cal O _{\Phi,\emptyset}(\germ sl _n)$ is no longer semisimple. The indecomposable modules of this category have been studied by Grantcharov and Serganova \cite{GS07} and by Mazorchuk and Stroppel \cite{MS10}.
\end{rema}


\subsection{Type $C$ case}\label{ssec:semiC}


In this section $\germ g =C_{n}$. Let $\theta$ be a proper subset of $\Phi$. We keep the notations of part \ref{ssec:typeC}.

\subsubsection{Case $\Phi\setminus\theta=\{\alpha\}$}\label{sssec:semiCA}

In this case we have the following Dynkin diagram:
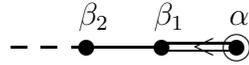
\setlength{\unitlength}{1cm}
\begin{center}
\begin{picture}(5,1.2)
\put(1,.5){\line(1,0){.2}}
\put(1.4,.5){\line(1,0){.2}}
\put(1.8,.5){\line(1,0){.2}}
\put(2,.5){\circle*{.2}}
\put(2,.5){\line(1,0){1}}
\put(3,.5){\circle*{.2}}
\put(3,0.45){\line(1,0){1}}
\put(3,.55){\line(1,0){1}}
\put(4,.5){\circle*{.2}}
\put(4,.5){\circle{.35}}
\put(1.9,.8){$\beta_2$}
\put(2.9,.8){$\beta_1$}
\put(3.9,.8){$\alpha$}
\put(3.4,.4){$<$}
\end{picture}
\end{center} 

From theorem \ref{thmpcpl} we know that the simple modules in the category $\cal O _{\Phi,\theta}$ are the $$M_a=M(-1,\ldots ,-1,a),$$ with $a\in \C\setminus\Z$. We prove:

\begin{theo}\label{thmextC}
In the category $\cal O _{\Phi,\theta}$, we have $Ext^1(M_b,M_a)=\{0\}$.
\end{theo}

\begin{rema}
We shall prove the following equivalent statement:

Any exacte sequence $$(S_{a,b}) : 0\rightarrow M_a\rightarrow M\rightarrow M_b\rightarrow 0$$ with $M\in\cal O _{\Phi,\theta}$ splits.
\end{rema}

\begin{proof} Let us denote by $c$ the cocyle corresponding to the exact sequence $(S_{a,b})$. Note first that the sequence splits as a sequence of $\germ l _{\theta}$-modules since the restriction condition of category $\cal O _{\Phi,\theta}$ expresses that $M_{a}$, $M_{b}$ and $M$ are semisimple $\germ l _{\theta}$-modules. In other words, the cocyle $c$ vanishes on $\germ l _{\theta}$. To prove the theorem we have to show it vanishes on the whole Lie algebra $\germ g $. Thanks to the cocyle relation, we only need to show that $c(X_{\pm\alpha})=0$.

\begin{enumerate}
\item {\bf First suppose that $M_a\not\cong M_b$}. From the explicit action of $\germ h $ on $M_{a}$ and $M_{b}$, we obtain $Supp(M_a)\cap Supp(M_b)=\emptyset$. Indeed, if $x(k)\in M_{a}$ has the same weight as $y(k')\in M_{b}$, we would have 
$$\left\{ \begin{array}{ccc}
k_{1}-k_2 & = & k'_{1}-k'_2\\
\vdots & = & \vdots\\
k_{n-2}-k_{n-1} & = & k'_{n-2}-k'_{n-1}\\
k_{n-1}-a-k_n-1 & = & k'_{n-1}-b-k'_n-1\\
a+k_n+\frac{1}{2} & = & b+k'_n+\frac{1}{2}
\end{array}\right.$$
The unique solution of this system is given by $k_i=k'_i$ for any $i\leq n-1$ and $b=a+k_n-k'_n$. But we must have $k_1+\cdots +k_n\in 2\Z$ and $k'_1+\cdots +k'_n\in 2\Z$. Since $k_i=k'_i$ for $i\leq n-1$, we must have $k_n-k'_n\in 2\Z$. Hence we have $b\in a+2\Z$. This means that the vector $y(k')\in M_b$ also appears in $M_a$. Then Lemire's correspondence \cite{Le68} would imply that $M_a\cong M_b$, which contradicts our assumption.

Now let $x\in M_{b}$ be a weight vector. From the cocycle relation and the fact that $c(H)=0$ for $H\in \germ h \subset \germ l _{\theta}$, we get that $c(X_{\pm\alpha})x$ is a weight vector of $M_{a}$, having the same weight as $X_{\pm\alpha}x$ which is non zero since the action of $X_{\pm\alpha}$ on $M_{b}$ is cuspidal. As $Supp(M_a)\cap Supp(M_b)=\emptyset$, this is impossible unless $c(X_{\pm\alpha})x=0$. Thus we proved that $c(X_{\pm\alpha})=0$, as asserted.

\item {\bf Suppose now that $M_a\cong M_b$}. Recall the action of $X^-=X_{-\alpha}$ and $X^+=X_{\alpha}$ on $M_a$ :
$$\left\{\begin{array}{ccc}
X^+\cdot x(k) & = & \frac{1}{2}x(k+2\epsilon_n)\\
X^-\cdot x(k) & = & -\frac{1}{2}(a+k_n)(a+k_n-1)x(k-2\epsilon_n)
\end{array}\right.$$
Proposition \ref{propNM} implies that $M_a$ splits into a direct sum of simple cuspidal $\germ l $-modules. Using the action of $X^-$ and $X^+$ we remark that the vectors $x(k')$ belonging to the module $\cal U (\germ l )x(k)$ satisfy $k'_i=k_i$ for $i\not= n$. Thus the module $\cal U (\germ l )x(k)$ does not depend upon $k_{n}$. Proposition \ref{prop:GS} now gives the cocycle on each $\cal U (\germ l )x(k)$: $c(X^-)=0$ and $c(X^+)=b(k)\times {\left(X^-\right)}^{-1}$ for some $b(k)\in \C$. We should remember from the previous remark that $b(k)$ does not depend upon $k_{n}$.

Let $\beta \in \langle \theta\rangle^+$ be such that $\alpha+\beta\in\cal R $. We write $\beta=\beta_1+\cdots +\beta_i$. Then we have $[X^+,X_{\beta}]=-X_{\alpha+\beta}$ and $[-X_{\alpha+\beta},X_{-\beta}]=2X^+$. Using the cocycle relation we get:
$$2c(X^+)=c([[X^+,X_{\beta}],X_{-\beta}])=[[c(Y),X_{\beta}],X_{-\beta}],$$ since $c(X_{\pm \beta})=0$ as $X_{\pm \beta}\in \germ l _{\theta}$. This is an identity in $End_{\C}(M_a)$. Let us apply it to the vector $x(k)\in M_a$. Recall that: 
$$\left\{\begin{array}{rcl}
X_{\beta}\cdot x(k) & = & k_{n-i}(a+k_n)x(k+\epsilon_{n-i}-\epsilon_n)\\
X_{-\beta}\cdot x(k) & = & x(k-\epsilon_{n-i}+\epsilon_n)
\end{array}\right.$$

Using the above action of $X^-$ we get
$$\left(X^-\right)^{-1}x(k)=-\frac{2}{(a+k_n+2)(a+k_n+1)}x(k+2\epsilon_n).$$ Therefore we have 
\begin{equation}\label{eqqqCC}
2c(X^+)x(k)=-\frac{4b(k)}{(a+k_{n}+2)(a+k_n+1)}x(k+2\epsilon_n).
\end{equation}
Hence: 
\begin{equation}\label{1deplus}
[c(X^+),X_{\beta}]x(k)=\frac{2k_{n-i}}{a+k_n+1}(b(k)-b(k+\epsilon_{n-i}))x(k+\epsilon_{n-i}+\epsilon_n)
\end{equation}
and thus
\begin{multline}\label{eqqqCCC}
[[c(X^+),X_{\beta}],X_{-\beta}]x(k) = \frac{2(k_{n-i}-1)}{a+k_n+2}(b(k-\epsilon_{n-i})-b(k))x(k+2\epsilon_{n})\\
  -\frac{2k_{n-i}}{a+k_n+1}(b(k)-b(k+\epsilon_{n-i}))x(k+2\epsilon_{n}).
\end{multline}
Equaling equations \eqref{eqqqCC} and \eqref{eqqqCCC}, we finally get:
\begin{multline}\label{eqqqCCCC}
2b(k)= k_{n-i}(a+k_n+2)(b(k)-b(k+\epsilon_{n-i}))\\
  -(k_{n-i}-1)(a+k_n+1)(b(k-\epsilon_{n})-b(k)).
\end{multline}
Let $k$ be such that $k_{n-i}=0$. Then the equation \eqref{1deplus} implies that $[c(X^+),X_{\beta}]x(k)=0$. Therefore $c(X^+)x(k)$ should be zero. Thus we have $b(k)=0$. If $k_{n-i}=-1$, then equation \eqref{eqqqCCCC} gives $3(a+k_{n}+2)b(k)=2(a+k_{n}+1)b(k-\epsilon_{n})$. Since $b(k)$ does not depend upon $k_{n}$, we must have $b(k)=0=b(k-\epsilon_{n})$. Now we show by induction on $k_{n-i}\leq 0$ using equation \eqref{eqqqCCCC} that $b(k-p\epsilon_{n-i})=0$ for any non negative integer $p$.  Thus the cocycle $c$ is zero, as asserted.
\end{enumerate}

\end{proof}

\begin{coro}\label{corextC}
Let $\theta \subset \Phi$ be such that $\Phi\setminus\theta=\{\alpha\}$. Then the category $\cal O _{\Phi,\theta}$ is semisimple.
\end{coro}

\subsubsection{Case $\germ l '=C_{j}$}\label{sssec:semiCC}

Theorem \ref{thmpcpl} implies that the simple modules of the category $\cal O _{\Phi, \theta}$ are modules of degree $1$. Then proposition \ref{propNM} asserts that these modules split into a direct sum of simple cuspidal $\germ l $--modules. Moreover any $M$ in the category $\cal O _{\Phi,\theta}$ can be seen as a cuspidal $\germ l $-module by restriction (this is the cuspidality condition). Applying theorem \ref{thmBKLM}, we get that as a $\germ l $-module $M$ is semisimple.

On the other hand, the restriction condition of category $\cal O _{\Phi,\theta}$ implies that $M$ is a semisimple $\germ l _{\theta}$-module. Hence $M$ is semisimple as a $\germ g $-module. Therefore we proved

\begin{coro}\label{corextCf}
Let $\theta \subset \Phi$ be such that $\Phi\setminus\theta=C_j$. Then the category $\cal O _{\Phi,\theta}$ is semisimple.
\end{coro}

\begin{rema}
If $\theta=\Phi$, the category $\cal O _{\Phi,\Phi}$ is semisimple by definition. On the other hand, if $\theta=\emptyset$, the category $\cal O _{\Phi,\emptyset}$ is also semisimple according to the theorem \ref{thmBKLM}. Thus if $\germ g =C_{n}$, then the category $\cal O _{\Phi,\theta}$ is semisimple for any subset $\theta$ of $\Phi$.
\end{rema}


\appendix

\section{Case $\germ g =A_{3}$}\label{append}


Let $\germ g =\germ sl _4(\C)$. We described the category $\cal O _{\Phi,\theta}(\germ g )$ in all cases except when $\theta=\{e_2,e_3\}$. We handle this case here. It corresponds to the following Dynkin diagram:

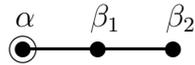
\setlength{\unitlength}{1cm}
\begin{center}
\begin{picture}(4,1.2)
\put(1,.5){\line(1,0){1}}
\put(1,.5){\circle*{.2}}
\put(2,.5){\circle*{.2}}
\put(2,.5){\line(1,0){1}}
\put(3,.5){\circle*{.2}}
\put(1,.5){\circle{.35}}
\put(.9,.8){$\alpha$}
\put(1.9,.8){$\beta_1$}
\put(2.9,.8){$\beta_{2}$}
\end{picture}
\end{center}

We shall use the notations introduced in section \ref{sec:A1An}. Recall we set $\germ l :=\germ l _{\Phi\setminus\theta}$. For simplicity, we set $X^-:=X_{-\alpha}$ and $X^+=X_{\alpha}$. Since the semisimple Lie algebra $\germ l '$ is isomorphic to $\germ sl _2(\C)$ we know that the simple modules in $\cal O _{\Phi,\theta}$ are the $L(C)$ where $C$ is a simple cuspidal $\germ l $-module, isomorphic to some $N(a_{1},a_{2})$ as a $\germ l '$-module, with $a_{1}$ and $a_{2}$ non integer complex numbers. Recall that we denoted by $V(C)$ the corresponding generalised Verma module and by $p$ the natural projection $p:V(C)\rightarrow L(C)$.

The center of $\germ l $ is two dimensional and generated by $H_1:=H_{\alpha}+2H_{\beta_1}$ and $H_2:=H_{\beta_2}$. We denote by $c(k)$ the action of $H_{\beta_{1}}$ on $x(k)\in C$. We have seen in lemma \ref{lemA12} that $c(k)=c+a_{2}-k$ with $c=0$ or $c=-1-a_{1}-a_{2}$. As $H_{2}$ is in the center of $\germ l $, it acts on $C$ by some constant that we shall denote by $d$. As in corollary \ref{corA12}, we prove that if $d=0$ then the module $L(C)$ is isomorphic to $N(a_{1},a_{2},0,0)$ if $c=0$ or to $N(-1-a_{2},-1-a_{1},0,0)$ if $c=-1-A$.

{\bf In what follows we assume that $d\not=0$}. This condition implies that $p(X_{-\beta_2}\otimes x(b))\not=0$. Indeed, we have $X_{\beta_2}\cdot X_{-\beta_2}\otimes x(k)=H_2\otimes x(k)=d\times 1\otimes x(k)\not=0$. Then we apply proposition \ref{GVM} to conclude. We shall find the possible values for $d$ in order the module $L(C)$ be in category $\cal O _{\Phi,\theta}$ in the spirit of section \ref{sec:A1An}.

From lemma \ref{lem7}, there are two complex numbers $\eta_{1}(k)$ and $\eta_{2}(k)$ such that
\begin{multline}\label{ann:EQ}
p(X_{-\alpha-\beta_1-\beta_2}\otimes x(k) = \eta_1(b)p(X_{-\beta_2}X_{-\beta_1}\otimes x(k-1))\\
+\eta_2(b)p(X_{-\beta_1}X_{-\beta_2}\otimes x(k-1)).
\end{multline}
We apply the vector  $X_{\alpha+\beta_1+\beta_2}$ to equation \eqref{ann:EQ}. We get:
$$p(H_{\alpha+\beta_1+\beta_2}\otimes x(k)) = \eta_1(k)p(X^+\otimes x(k-1)),$$ which gives us the following equation:
\begin{align}
c+d+a_1+k= & (a_2-k+1)\eta_1(k).\label{eqA13a}
\end{align}
We apply then the vector $X_{\beta_2}$ to equation \eqref{ann:EQ}. We have
$$p(X_{-\alpha-\beta_1}\otimes x(k)) = \eta_1(k)p(H_2X_{-\beta_1}\otimes x(k-1))+\eta_2(k)p(X_{-\beta_1}H_2\otimes x(k-1)),$$ which gives together with lemma \ref{lemA12}:
\begin{align}
\eta(k) = & (d+1)\eta_1(k)+d\eta_2(k).\label{eqA13b}
\end{align}
Finally we apply $X_{\beta_1+\beta_2}$ to equation \eqref{ann:EQ}. We get
$$p(X^-\otimes x(k)) = \eta_1(k)p(H_{\beta_1}\otimes x(k-1))+\eta_2(k)p(-H_2\otimes x(k-1)),$$ from which we obtain
\begin{align}
a_1+k = & (c+a_2-k+1)\eta_1(k)-d\eta_2(k).\label{eqA13c}
\end{align}
From equations  \eqref{eqA13a}, \eqref{eqA13b} and \eqref{eqA13c} we find the following values:
$$\left\{\begin{array}{cc}
\eta_1(k) = & -\frac{c+a_2-k+2}{a_2-k+1},\\
\eta_2(k) = & \frac{c+a_2-k+1}{a_2-k+1},\\
d = & -2-A-2c.
\end{array}\right.$$
Note in particular that $d$ is entirely determined by $A=a_{1}+a_{2}$ and $c$.

Conversely, we have to show that given $c\in\{0,-1-A\}$ and $d=-2-A-2c$ the corresponding module $L_{c,d}(a_1,a_2)$ belongs to the category $\cal O _{\Phi,\theta}$. We prove it for the case $(c,d)=(0,-2-A)$, the other one being analogous. Of course we only need to show that the module $L(C)$ satisfies the restriction condition. First PBW theorem implies that $L(C)$ is generated by the following weight vectors:
$$p(X_{-\beta_1-\beta_2}^{m_1}X_{-\beta_2}^{m_2}X_{-\beta_1}^{m_3}X_{-\alpha-\beta_1-\beta_2}^{m_4}X_{-\alpha-\beta_1}^{m_5}\otimes x(k)).$$
Thanks to lemma \ref{lemA12} and equation \eqref{ann:EQ}, this reduces to the following vectors only: 
$$p(X_{-\beta_1-\beta_2}^{m_1}X_{-\beta_2}^{m_2}X_{-\beta_1}^{m_3}\otimes x(k)).$$
This proves that $L(C)$ splits into a direct sum of $\germ l _{\theta}$-modules, namely the modules $M_k:=p(\cal U (\germ l _{\theta})\otimes x(k))$. 

Thus it remains to prove that $M_{k}$ is a simple highest weight module. From lemma \ref{lem2} we already know that $M_{k}$ is a highest weight module, generated by the highest weight vector $p(1\otimes x(k))$. Its $\germ l '_{\theta}$-highest weight is $\lambda=(a_2-k)\omega_1+(-2-A)\omega_2$ where $\omega_{i}$ are the fundamental weights of $\germ l '_{\theta}$. Let $V_{k}$ be the $\germ l '_{\theta}$-Verma module with highest weight $\lambda$. Using a theorem of Bernstein-Gelfand-Gelfand (see for instance \cite[thm 5.1]{Hu08}), we remark that $V_{k}$ is simple (and therefore isomorphic to $M_{k}$ by the universal property of Verma modules) if and only if $A\not\in\Z_{<-1}$.

So it only remains to work out the case $A\in\Z_{<-1}$. In this case, Bernstein-Gelfand-Gelfand's theorem \cite[thm 5.1]{Hu08} together with the Kazdhan-Lusztig conjecture for rank two Lie algebras (see for instance \cite[chapitre 8]{Hu08}) show that $V$ admits a unique simple submodule $L(\mu)$ with $\mu=s_{_{\beta_{2}}}\cdot \lambda$. We shall see $V_{k}$ as a $\germ l '_{\theta}$-submodule of $V(C)$. Then it is easy matter to check that the submodule $L(\mu)$ is generated by $X_{-\beta_2}^{-A-1}\otimes x(k)$. On the other hand, one shows by straightforward computation that $X_{-\beta_2}^{-A-1}\otimes x(k)$ is annihilated by the action of $X_{\alpha+\beta_1}$ and $X_{\alpha+\beta_1+\beta_2}$. Using proposition \ref{GVM}, we conclude that $p(X_{-\beta_2}^{-A-1}\otimes x(b))=0$. But now by the universal property of the Verma module $V_{k}$ there is a surjective map from $V_{k}$ onto $M_{k}$. As we just shown that $p(L(\mu))=0$, we get a surjective map from $V_{k}/L(\mu)$ onto $M_k$. As $L(\mu)$ is the unique submodule of $V_{k}$, the quotient $V_{k}/L(\mu)$ is simple and so is $M_{k}$. This completes the proof of the following

\begin{prop}
\begin{enumerate}
\item The simple modules in $\cal O _{\Phi,\theta}$ are the modules $L_{c,d}(a_1,a_2)$ with $a_1,a_2\in \C\setminus \Z$, $c\in \{0,-1-a_1-a_2\}$ and $d\in \{0,-2-a_1-a_2-2c\}$.
\item If $d=0$, the module $L_{c,d}(a_1,a_2)$ is of degree $1$, isomorphic to some module $N(a'_1,a'_2,0,0)$, with $a'_{1}$ and $a'_{2}$ non integer complex numbers.
\end{enumerate}
\end{prop}

\begin{rema}
If $d\not\in \Z$, the module $L_{c,d}(a_1,a_2)$ is not of finite degree. Indeed, in this case $A\not\in \Z$. Therefore the previous proof implies that the simple $\germ l _{\theta}$-modules $M_{k}$ are isomorphic to Verma modules, which are obviously not of finite degree.
\end{rema}

\bibliographystyle{plain}
\bibliography{biblio}

\end{document}